\documentclass[a4paper,english,12pt]{amsart}
\usepackage{amsmath}
\usepackage{amssymb}
\usepackage{amsthm}
\usepackage{stmaryrd}
\usepackage{xcolor}
\usepackage{tikz}
\usepackage{tikz-cd}
\usepackage{hyperref}
\newcommand{\doots}{,\dots,}
\newcommand{\NN}{\mathbb N}

\newcommand{\CCC}{\mathfrak C}
\newcommand{\CZN}{\CCC_{\ZZ/p^e\ZZ}(\leq n)}
\newcommand{\WW}{\mathbb W}

\newcommand{\QQ}{\mathbb Q}
\newcommand{\OO}{\mathcal O}

\newcommand{\Ov}{\OO_v}

\newcommand{\CC}{\mathbb C}

\newcommand{\RR}{\mathbb R}

\newcommand{\ZZ}{\mathbb Z}
\newcommand{\FF}{\mathbb F}

\newcommand{\Fv}{F_v}

\newcommand{\Ga}{\mathbb G_a}

\newcommand{\AAA}{\mathbb A}
\newcommand{\AAF}{\mathbb A_F}

\newcommand{\Zpez}{\ZZ/p^e\ZZ}

\newcommand{\Wphim}{(\WW_{\phi}^{\leq m})^*}

\DeclareMathOperator{\Aff}{Aff}

\DeclareMathOperator{\Spec}{Spec}
\DeclareMathOperator{\Frob}{Frob}

\DeclareMathOperator{\Gal}{Gal}

\DeclareMathOperator{\Hom}{Hom}

\DeclareMathOperator{\ord}{ord}

\DeclareMathOperator{\type}{type}

\DeclareMathOperator{\pr}{pr}

\DeclareMathOperator{\Imm}{Im}

\DeclareMathOperator{\tr}{tr}
\DeclareMathOperator{\cc}{\mathfrak c}

\DeclareMathOperator{\Id}{Id}

\DeclareMathOperator{\cond}{cond}

\DeclareMathOperator{\disc}{disc}

\DeclareMathOperator{\wdeg}{wdeg}
\DeclareMathOperator{\fppf}{fppf}

\DeclareFontFamily{U}{wncy}{}
\DeclareFontShape{U}{wncy}{m}{n}{<->wncyr10}{}
\DeclareSymbolFont{mcy}{U}{wncy}{m}{n}
\DeclareMathSymbol{\Sh}{\mathord}{mcy}{"58}
\newtheorem{conj}{Conjecture}[subsubsection]
\newtheorem{thm}[conj]{Theorem}
\newtheorem{lem}[conj]{Lemma}
\newtheorem{cor}[conj]{Corrolary}
\newtheorem{prop}[conj]{Proposition}
\newtheorem{rem}[conj]{Remark}
\newtheorem{exam}[conj]{Example}

\newtheorem{mydef}[conj]{Definition}

\title{Counting torsors for wild abelian groups}
\author[Ratko Darda]{Ratko Darda}
\address{Faculty of Engeenering and Natural Sciences, Sabanc\i\hspace{0,05cm} University, Istanbul, Turkey}
\email{ratko.darda@gmail.com, ratko.darda@sabanciuniv.edu}
\author[Takehiko Yasuda]{Takehiko Yasuda}
\address{Department of Mathematics\\ Graduate School of Sciences\\ the University of Osaka, Osaka, Japan}
\email{yasuda.takehiko.sci@osaka-u.ac.jp}
\begin{document}
	\maketitle
	\begin{abstract}
	Let~$F$ be a global field of characteristic $p>0$ and $G$ a finite abelian $p$-group. In this paper we treat the question of counting $G$-torsors over~$F$ for certain heights developed in \cite{dardayasuda2025}.
\end{abstract}
	\section{Introduction}
\subsection{} The distribution of field extensions is an actively studied problem in number theory. A guiding principle is provided by Malle's conjecture \cite{Malle}, which predicts the number of Galois extensions of a number field of bounded discriminant or discriminant-like height. The conjecture has been established in a number of instances \cite{Wright}, \cite{davenportheilbronn}, \cite{densityquartic}, \cite{densityquintic}, etc. Wright's result~\cite{Wright} on the number of abelian extensions of bounded discriminant is valid over a global field for any abelian group~$G$ satisfying that the order of group is coprime to the characteristic of the field, if the characteristic is positive. 
In this paper, we would like to treat the question of counting abelian extensions, or more generally $G$-torsors, {\it without} assuming the condition on the order of the group. Such problems 
have attracted increasing attention in recent research. Lagemann has found asymptotic formulas for the number of abelian extensions of bounded conductor \cite{LAGEMANN2015288}.  Gundlach proved an analogous result for abelian extensions of bounded so-called Artin--Schreir conductor \cite{gundlach2024countingabelianextensionsartinschreier}, while Gundlach--Seguin \cite{gundlach2024asymptoticsextensionssimplemathbb} proved it for certain 2--step nilpotent  extensions of bounded Artin--Screir conductor. 

\subsection{}Let us state the contributions of this paper. Let~$F$ be a global field of characteristic~$p>0$. Let~$G$ be a finite abelian $p$-group. 
We denote by $BG\langle F\rangle$ the set of isomorphism classes of $F$-points of the stack~$BG$, or in other words, the set of isomorphism classes of $G$-torsors over~$F$. 
In \cite{dardayasuda2025}, we have defined heights on wild isotropic $F$-Deligne--Mumford stacks, including the stack~$BG$. 
A height~$H$ possesses $a$- and $b$-invariants, denoted by $a(H)$ and $b(H)$, which are expected to appear in the asymptotic formula for counting, as speculated in \cite[Main Speculation]{dardayasuda2025}. This generalizes the tame situation, previously considered in \cite[Conjecture 9.10]{dardayasudabm}, which generalizes both the Batyrev--Manin conjecture on the distributions of rational points of bounded height on varieties and the Malle's conjecture. We prove that: 
\begin{thm}[Corollary~\ref{ainvariantcorrect}] \label{sheight}
Let $H:BG\langle F\rangle\to\RR_{>0}$ be a height satisfying the conditions stated in  Definition~\ref{suitable}. For any $\varepsilon>0$, we have that $$B^{a(H)}\ll\#\{x\in BG\langle F\rangle|\hspace{0,1cm}H(x)\leq B\}\ll B^{a(H)+\varepsilon}.$$
\end{thm}
The discriminant~$H_{\disc}$ is an example of such height (Lemma~\ref{discsuitable}). It follows from result of Lagemann~\cite[Theorems 2.3 and 2.4]{LAGEMANN2015288} that
$$B^{a}\ll\#\{x\in BG\langle F\rangle|\hspace{0,1cm}H_{\disc}(x)\leq B\}\ll B^{ad},$$ for some~$a>0$ and for some number~$d$ which satisfies $d>1$. Our bounds are hence necessarily stronger than his. With some additional hypothesis, we prove that:
\begin{thm}[Theorem~\ref{strsuitthm}]\label{sssheights}
Let $H:BG\langle F\rangle\to\RR_{>0}$ be a height satisfying the conditions stated in  Definition~\ref{suitable} and in Definition~\ref{strongsuitdef}. One has that $$\#\{x\in BG\langle F\rangle|\hspace{0,1cm}H(x)\leq B\}\asymp B^{a(H)}\log(B)^{b(H)-1}.$$
\end{thm}
An example of such a height is given by the conductor (Lemma~\ref{condss}), which falls within the scope of Lagemann’s aforementioned result~\cite{LAGEMANN2015288}. Further examples are provided in Remark~\ref{newheights}, where an interesting phenomenon is observed: contrary to the tame case, the $b$-invariant can be arbitrarily large.  The tame case for general heights has been treated in \cite{Wood_2010} and \cite{commcase}.
Another source of height functions is linear representations. In \ref{ExamplesFromReps}, we will show that when $G=\ZZ/p\ZZ$, the height function associated to some linear representation satisfy the assumptions of the last theorem.

\subsection{} Our proof of Theorem~\ref{sssheights} uses realization of a finite quotient of $BG\langle F\rangle$ as a discrete subgroup of adelic torsors $BG\langle \AAA_F\rangle$. The conditions imposed on $H$ imply that the height balls in $BG\langle \AAA_F\rangle$ have finite volumes. We deduce the height zeta function $Z(s):=\sum_{x\in BG\langle F\rangle}H(x)^{-s}$ converges absolutely and is holomorphic for $\Re(s)\gg 0$. The heights can be extended on $BG\langle \AAF\rangle$. 
 Then we compare $Z(s)$ and $\int_{BG\langle \AAF\rangle}H^{-s}\mu$, where $\mu$ is a Haar measure. 
 If~$H$ satisfies only the weaker conditions, that is the ones of Theorem~\ref{sheight}, then~$H$ is modified so that the new height satisfies the stronger conditions of Theorem~\ref{sssheights} and that it does not differ ``too much'' from the original height. To show that the discriminant is height satisfying the weaker conditions, we use conductor discriminant formula in Section~\ref{Discriminantloci}. In Section~\ref{msfg} we study the space~$\Delta_G$, introduced by Tonini and the second named author in  \cite{tonini-yasuda} and \cite{tonini-yasuda2}, which over a finite field~$k$ parametrizes $G$-torsors over $k((t))$. 
The analysis of the number of points of bounded heights occupies Section~\ref{npbh}.
\subsubsection{} By~$p$ we will denote a prime number and by~$q$ some of its powers. For primer power~$d$ we will write $\FF_d$ for a fixed finite field of cardinality~$d$. For a ring~$A$ we write $A\llparenthesis t\rrparenthesis$ for the localization at~$t$ of the ring of power series $A[[t]]$. In particular, if $A$ is a field, then $A\llparenthesis t\rrparenthesis$ is the field of Laurent series~$A((t))$. For a finite group~$G$, we denote by~$BG$ the classifying stack of~$G$. By $\NN$ we mean the set $\{1,2,\dots\}$ and by $\NN_0$ we mean $\NN\cup\{0\}$. Assuming~$p$ is known prime, we denote by $\NN'$ the set $\NN-p\NN$. 

\subsubsection*{Acknowledgements}
Ratko Darda has received a funding from the European Union’s Horizon 2023 research and innovation programme under the Maria Skłodowska-Curie grant agreement 101149785. Takehiko Yasuda was supported by JSPS KAKENHI Grant Numbers JP21H04994, JP23H01070 and JP24K00519. 

\section{Moduli stack of formal $G$-torsors}\label{msfg}
Let $p$ be a prime and let~$G$ be a finite abelian $p$-group. The goal of this chapter is to study the stack~$\Delta_G$ from \cite{tonini-yasuda}, with the focus on the group structure. In \cite{tonini-yasuda}, it is established that $\Delta_G$ is an ind-Deligne--Mumford stack and we prove that it is an ind-group Deligne--Mumford stack.
	 \subsubsection{} We define the object of our study.
	\begin{mydef}\label{deltaggpstack}
		Let $G$ be a finite abelian $p$-group. We denote by $\Delta_G$ the following category fibered in groupoids over affine $\FF_q$-schemes
		$$\Delta_G:\Spec(A)\mapsto BG(\Spec(A\llparenthesis t\rrparenthesis)).$$
	\end{mydef}
	\begin{lem}
		One has that~$\Delta_G$ is a commutative group stack (called {\it strictly commutative Picard stack} in \cite[Expos\' e XVIII, Section 1.4]{Dualite}) on the category of affine $\FF_q$-schemes endowed with fpqc topology.
	\end{lem}
	\begin{proof}Recall by~\cite[Corollary 4.16]{tonini-yasuda} that~$\Delta_G$ is an fpqc stack on the category of affine $\FF_q$-schemes. We define a functor $+_{\Delta_G}:\Delta_G\times\Delta_G\to\Delta_G$ by setting it for an affine $\FF_q$-scheme $\Spec(A)$ to be $+_{BG}\circ (\Spec(A)\mapsto \Spec(A\llparenthesis t \rrparenthesis))$, where $+_{BG}:BG\times BG\to BG$ is the addition in the commutative group stack~$BG$. We define a commutativity functor $c_{\Delta_G}:x_0+_{\Delta_G}y_0\xrightarrow{\sim}y_0+_{\Delta_G}x_0$ by setting it to be the commutativity functor $c_{BG}:x+_{BG}y\xrightarrow{\sim}y+_{BG}x$ in the group stack~$BG$ precomposed with the functor $\Spec(A)\mapsto \Spec(A\llparenthesis t \rrparenthesis)$. Similarly, we define an associativity functor $a_{\Delta_G}:(x_0+_{\Delta_G}y_0)+_{\Delta_G}z_0\xrightarrow{\sim}x_0+_{\Delta_G}(y_0+_{\Delta_G}z_0)$ by setting it to be associativity functor $a_{BG}:(x+_{BG}y)+_{BG}z\xrightarrow{\sim}x+_{BG}(y+_{BG}z)$ precomposed with the functor $\Spec(A)\mapsto \Spec(A\llparenthesis t \rrparenthesis)$. We need to verify that the addition, commutativity and associativity satisfy properties from \cite[Expos\'e XVIII, Paragraph 1.4]{Dualite}. This comes to verify that for any~$A$, one has that $(\Delta_G(A),+_{\Delta_G}(A),c_{\Delta_G}(A),a_{BG}(A))$ is a strictly commutative Picard category (\cite[Definition 1.4.2]{Dualite}). But by definition, one has that 
\begin{multline*}(\Delta_G(A),+_{\Delta_G}(A),c_{\Delta_G}(A),a_{\Delta_G}(A))\\=(BG(A\llparenthesis t \rrparenthesis),+_{BG}(A\llparenthesis t\rrparenthesis)),c_{BG}(A\llparenthesis t \rrparenthesis), a_{BG}(A\llparenthesis t \rrparenthesis)),
	\end{multline*}
which is a strictly commutative Picard category because $BG=0/G$ is a commutative group stack (defined by the chain complex \cite[Paragraph 1.4.11]{Dualite} having the only non zero term at degree~$-1$ which is~$G$).
\end{proof}
	\subsection{Witt vectors}  In this section we review Witt vectors. For $n\geq 1$, let~$W_n$ be the group scheme of Witt vectors of length~$e$.   \subsubsection{}We evaluate degrees of polynomials for the addition of Witt vectors. 
	\begin{mydef}For a non-zero polynomial~$P$ in $X_0\doots X_n$, we say that it is $(X)$-weighted homogeneous of weighted degree~$d$ if it is homogeneous when $\deg(X_i)=p^i$ and of degree~$d$.  For a non-zero polynomial~$P$ in $X_0,Y_0\doots X_n,Y_n$, we say that is $(X,Y)$-weighted homogeneous of degree~$d$ is it is homogeneous when $\deg({X_i})=\deg(Y_i)=p^{i}$ and the degree is~$d$.   Similarly, for a non-zero polynomial~$P$ in $\mathbf X^1\doots \mathbf X^d$ where each $\mathbf X^u$ are variables $X^u_0\doots X^u_n$, we say that it is $(\mathbf X^1\doots \mathbf X^d)$-homogeneous when $\deg(X^u_i)=p^i$ for all~$u$ and~$i$, and the degree is~$d$. Specifically, the polynomial $\mathbf 0$ will be considered weighted homogeneous of any degree. In either of the above cases, write $\wdeg(P)$ for the weighted degree of~$P$.
	\end{mydef}
	\begin{exam}
		\normalfont One has that the polynomial $$\Phi_n(X_0\doots X_n)=\sum_{i=0}^np^iX_i^{p^{n-i}}$$is $(X)$-weighted homogeneous of the degree~$p^n$.
	\end{exam}
	Let $ S_0, S_1\dots $ be the polynomials for addition of Witt vectors. The polynomial~$S_n$ is in variables $X_0,Y_0\doots X_{n}, Y_{n}$. 
	\begin{lem}\label{degwittaddition}
		One has that~$S_n$ is $(X,Y)$-homogeneous of degree~$p^n$.
	\end{lem}
	\begin{proof}
		One has that $\wdeg(S_0)=1=p^0.$ Suppose the claim is valid for some~$n$. One has that \begin{align*}\sum_{i=0}^{n}p^{i}(X_i^{p^{n-i}}+Y_i^{p^{n-i}})&=\Phi_{n}(X_0\doots X_{n})+\Phi_{n}(Y_0\doots Y_{n})\\&=\Phi_{n}(S_0(X_0,Y_0)\doots S_{n}(X_0,Y_0\doots X_n,Y_n))\\
			&=\sum_{i=0}^{n}p^{i}S_i(X_0,Y_0\doots X_i,Y_i)^{p^{n-i}}.
		\end{align*}
		Each monomial in $S_i(X_0,Y_0\doots X_i,Y_i)$ is of weighted degree $p^{i}$ by the induction hypothesis. It follows that each monomial in $S_i(X_0,Y_0\doots X_i, Y_i)^{p^{n-i}}$ is of weighted degree~$p^n$. It follows that $\wdeg(S_n)=p^n.$
	\end{proof}
	For $j=0,1,\dots$ we set $$T_j:=S_j-X_j-Y_j.$$ Then $T_0=0$, while as $X_j+Y_j$ is $(X,Y)$-homogeneous of weighted degree~$p^j$, we have that~$T_j$ is a polynomial in $X_0,Y_0\doots X_{j-1}, Y_{j-1}$ of weighted degree~$p^j$. For $j=0$ we define $I_0=-X_0$ and for $j\geq 1$ we define $$I_j:=-X_j-T_j(X_0, \doots X_{j-1},I_0(X_0)\doots  I_{j-1}(X_0\doots X_{j-1})).$$
	For an element $f=(f_j)_{j=0}^{e-1}\in W_e$, we verify that $$g=(I_0(f_0),\doots I_{j}(f_0\doots f_j))$$ is the additive inverse $g=-f$. Indeed, the $j$-th entry of $f+g$ is \begin{align*}S_j(f_0\doots f_{j},I_0(f_0)\doots I_j(f_0\doots f_j))\hskip-4,5cm&\\&=f_j+I_j(f_0\doots f_j)+ T_j(f_0\doots f_{j-1},I_0(f_0)\doots I_{j-1}(f_0\doots f_{j-1}))\\
		&=0.
	\end{align*}
	\begin{lem}\label{wittinverse}
		One has that $I_j$ is $(X)$-weighted homogeneous of weighted degree~$p^j$.
	\end{lem}
	\begin{proof}
		When $j=0$, one has that $I_0=-X_0$, which is $(X)$-homogeneous of weighted degree~$1$. Suppose the claim is valid for all non-negative integers no more than~$j$ for some $j\geq 0$. One has that $-X_{j+1}$ is $(X)$-weighted homogeneous of weighted degree~$p^{j+1}$. The polynomial~$T_{j+1}$ is $(X,Y)$-weighted homogeneous of weighted degree~$p^{j+1}$ and is independent of the variable $X_{j+1}$.  But by induction for any $0\leq k\leq j$ one has that $I_k(X_0\doots X_k)$ is $(X)$-weighted homogeneous of weighted degree~$p^{j+1}$. Hence, one has that $$T_{j+1}(X_0\doots X_{j-1},I_0(X_0)\doots I_j(X_0\doots X_{j-1}))$$ is $(X)$-weighted homogeneous of weighted degree~$p^{j+1}$ and is moreover independent of variable $X_{j+1}$. Thus $I_{j+1}$ is $(X)$-weighted homogeneous of weighted degree~$p^{j+1}$.
	\end{proof}
	\begin{lem}\label{wittbig}
		Let $d\geq 1$. The $j$-th coordinate of the addition $+^d:W_e^d\to W_e$ is given by a polynomial which is $(\mathbf X_1\doots \mathbf X_d)$-weighted homogeneous of weighted degree~$p^j$. For an element $k\in [0\doots p^{e}-1]$, the $j$-th coordinate of multiplication by~$k$ is given by a polynomial which is $(X)$-homogeneous of weighted degree~$p^j$.
	\end{lem}
	\begin{proof}
		The proof of both claims is immediate by using Lemma~\ref{degwittaddition} and the induction.
	\end{proof}
	\subsubsection{} In this paragraph, we recall the following fact.
	\begin{lem}\label{awnz} Let~$W$ be a group scheme which is isomorphic to an affine space. For every $\FF_q$-ring $A$, one has that $H^1_{\fppf}(A,W)=0$. In particular, this is valid for the group schemes~$W_n$.
	\end{lem}
	\begin{proof}
		It follows from Lazard's theorem \cite[Theorem 16.59]{MilneAlgebraicGroups} that~$W$ is split unipotent, which means that there exists series $$W=D_0\supset D_1\supset \cdots \supset D_k=e,$$
		with $D_j$ normal algebraic subgroup of~$D_{j-1}$ with $D_{j-1}/D_{j}\cong\Ga$ for any $j=1\doots k$. 
		Fix an $\FF_q$-algebra~$A$. By Hilbert 90 theorem, one has that $H^1_{Zar}(\Spec(A),\Ga)=H^1_{\fppf}(\Spec(A),\Ga)=0. $ For $j=k-1\doots 0$, we deduce an exact sequence
		$$H^1_{\fppf}(\Spec(A),D_{j+1})\to H^1_{\fppf}(\Spec(A), D_j)\to H^1_{\fppf}(\Spec(A),\Ga)=0.$$
		An induction argument proves the claim.
	\end{proof}
		\subsection{The stack $\Delta_G$}
\subsubsection{}Suppose that $G=\ZZ/p^e\ZZ$ for certain $e\geq 1$. Consider the fibered category $\widetilde{\mathfrak W}_e$ over the category of affine $\FF_q$-schemes, defined by $\widetilde{\mathfrak W}_e(A):=W_e(A\llparenthesis t \rrparenthesis)$ and such that for $x,y\in \widetilde{\mathfrak W}_e(A)$, one has $\Hom(x,y)=\{u\in W_e(A\llparenthesis t \rrparenthesis)|\hspace{0,1cm}x+\wp_e(u)=y\}$. The inverse of a morphism~$u$ is the morphism~$-u$. 
		\begin{lem}
			One has that~$\widetilde{\mathfrak W}_e$ is a commutative group stack and one has canonical equivalence $\Delta_{\ZZ/p^e\ZZ}\xrightarrow{\sim}\widetilde{\mathfrak W}_e$ of commutative group stacks.
		\end{lem}
		\begin{proof}
	 By \cite[Paragraph 1.4.12]{Dualite}, the exact sequence $$0\to \ZZ/p^e\ZZ\to W_e\xrightarrow{\wp_e} W_e\to 0 $$ induces an equivalence of Deligne--Mumford commutative group stacks $$B(\ZZ/p^e\ZZ)\xrightarrow{\sim}W_e/_{\wp_e}W_e,$$ where $/_{\wp_e}$ denotes that the quotient is with respect to homomorphism $\wp_e:W_e\to W_e$. From Lemma~\ref{deltaggpstack}, we deduce that $A\mapsto (W_e/\wp_eW_e)(A\llparenthesis t \rrparenthesis)$ is a commutative group stack. We also are given an equivalence $$\Delta_{\Zpez}= (A\mapsto B(\ZZ/p^e\ZZ)(A\llparenthesis t \rrparenthesis))\xrightarrow{\sim}(A\mapsto  (W_e/\wp_eW_e)(A\llparenthesis t \rrparenthesis) )$$ of commutative group stacks. By using Proposition~\ref{gogi}, we see that $A\mapsto (W_e/\wp_eW_e)(A\llparenthesis t \rrparenthesis)$ is equivalent to the category~$\widetilde{\mathfrak{W}}_{e}$ and that the latter category is a commutative group stack. 
		\end{proof}
		\begin{lem}\label{vot}
			\begin{enumerate}
				\item  Suppose that $u=(\sum_i u_{ij}t^{-i})\in W_e(A\llparenthesis t \rrparenthesis)$ is such that for some $i<0$ and some $j\in\{0\doots e-1\}$, one has that $u_{ij}\neq 0$. Then $\wp_e(u)$ is of the same form. 
				\item Suppose that $u$ is of the form as in (1) and $x=(\sum_{i\in\NN\cup\{0\}} x_{ij}t^{-i})_j$. Then $x+\wp_e(u)$ is of the form as in (1)
			\end{enumerate}
		\end{lem}
		\begin{proof}
			\begin{enumerate}
				\item 
				Let $j$ be the smallest index for which $\sum_i u_{ij}t^{-i}$ possesses an index $i_0<0$ with $u_{i_0j}\neq 0$ and choose~$v_p(i_0)$ as small as possible. Define $u_{rj}$ to be~$0$ whenever $r\in\QQ-\ZZ$. The formula for the $j$-th term for the addition is a polynomial in $X_0, Y_0\doots X_j, Y_j$, which is linear in $X_j$ and $Y_j$. Thus, for $i<0$, the coefficient in front of $t^{-i}$ of the $j$-th entry of $\wp_e(u)$ is $u^p_{(-i/p)j}-u_{(-i)j}$. Clearly, $u^p_{(-i_0/p)j}-u_{(-i_0)j}=-u_{(-i_0)j}\neq 0$. The statement is verified.
				\item By~(1), we have that $\wp_e(u)$ is of the form as in (1). Let $j$-th coordinate of  $\wp_e(u)$ be given by $\sum_i u_{ij}'t^{-i}$. Let $j$ be the smallest index for which $\sum_i u_{ij}'t^{-i}$ possesses an index $i_0<0$ with $u_{(i_0)j}'\neq 0$. The formula for the $j$-th entry for the addition of Witt vectors is linear in $X_j, Y_j$. Thus the coefficient in front of $t^{-i_0}$ is precisely $u_{(i_0)j}'\neq 0$.
			\end{enumerate}
		\end{proof}
		\begin{lem}\label{wedelta}
			We define ${\mathfrak W_e}$ by setting it for every ring~$A$ to be the subcategory of $\widetilde{\mathfrak W}_e(A)$ having for objects the elements  $(x_0\doots x_{e-1})\in W_e(A\llparenthesis t \rrparenthesis)$ satisfying that each~$x_j$ is of the form  $x_j=\sum_{i\geq 0} b_{ij}t^{-i}$ with sums being finite and having for morphisms between $x,y\in \mathfrak W_e$ the elements $u=(\sum_{i\geq 0}u_{ij}t^{-i})_j$ with sums being finite satisfying $x+\wp_e(u)=y$. One has that $\mathfrak W_e$ is a commutative group stack and that the canonical inclusion ${{\mathfrak W}_e}\hookrightarrow \widetilde{\mathfrak W}_e$ is an equivalence of commutative group stacks.
		\end{lem}
		\begin{proof}			
			{\it Group structure}.  The addition functor of $\widetilde{\mathfrak {W}}_e$ restricts to a functor $\mathfrak W_e\times \mathfrak W_e\to\mathfrak W_e$, because for any ring~$A$, one has that the sum of two elements in $\mathfrak W_e(A)$ is in $\mathfrak W_e(A)$. The commutativity functor of~$\widetilde{\mathfrak W}_e$ restricts to a functor $$(\mathfrak W_e\times \mathfrak W_e\to \mathfrak W_e)\to (\mathfrak W_e\times \mathfrak W_e\to \mathfrak W_e) $$ because it is given by the element~$(0\doots 0)$. The analogous claim is valid for the associativity functor. 
			We deduce that $\widetilde{\mathfrak W}_e$ is a group substack of $\widetilde{\mathfrak W}_e$. 
			
			{\it Essential surjectivity.}
			We prove that for every $x=(x_0\doots x_{e-1})\in W_e(A)$, one has that there exists $u\in W_e(A)$ and $y=(y_0\doots y_{e-1})\in{\widetilde{\mathfrak W}_e}(A)$ such that $y+\wp_e(u)=x.$ By \cite[Lemma 4.11]{tonini-yasuda}, there exists $u_0\in A\llparenthesis t \rrparenthesis$ and $y_0\in {\mathfrak W}(A)$ such that $y_0+\wp_1(u_0)=x_0$. Then $$x-\wp_e((u_0)_{k=0},(0)_{k=1\doots e-1})=x-(\wp_1(u_0)_{k=0},(0)_{k=1\doots e-1})=(y_0,x_1'\doots x_{e-1}'),$$for some $x_1'\doots x_{e-1}'$. Suppose that for $j\in\{0\doots e-2\}$, we have proven that the existance of $u^j:=((u_0\doots u_j),(0)_{k=j+1}^{e-1})$ such that $$x-\wp_e(u^j)=((y_0\doots y_{j}),(x_k'')_{k=j+1}^{e-1}) $$for some $x_k''$. Let $u_{j+1}$ be such that $x_{j+1}''-\wp_1(u_{j+1})=y_{j+1}$ for some $y_{j+1}\in {\mathfrak W^{\leq 0}_1}(A)$, which is possible thanks to \cite[Lemma 4.11]{tonini-yasuda}. Set $u^{j+1}=(u_k)_{k=0}^{j+1},(0)_{k=j+2}^{e-1}$. Then \begin{align*}x-\wp_e(u^{j+1})&=x-\wp_e(u^j)-\wp_e((0)_{k=0}^j,u_{j+1},(0)_{k=j+2}^{e-1} )\\
				&=x-\wp_e(u^j)- (0)_{k=0}^j,\wp_1(u_{j+1}),(u_k')_{k=j+2}^{e-1})\\
				&=((y_0\doots y_{j+1}), x_{k+1}'''\doots x_{e-1}''')
			\end{align*}
			for some $u_{j+2}'\dots u_{e-1}'$ and some $x_{j+2}'''\doots x_{e-1}'''$. The essential surjectivity is established.
			
			{\it Faithfulness.} The faithfulness is immediate, as $\Hom_{\mathfrak W_e}(x,y)$ is defined to be a subset of $\Hom_{\widetilde{\mathfrak W}_e}(x,y)$.
			
			{\it Fullness.} We need to show that if $x,y\in\mathfrak W_e$, then every morphism $u\in\Hom_{\widetilde{\mathfrak W}_e}(x,y)$ is of the form $(\sum_{i\geq 0}u_{ij}t^{-i})_j$. But this follows from Lemma~\ref{vot}(2).
		\end{proof}
		From now on, we will implicitly identify~${\mathfrak W}_e$ and~$\Delta_{\ZZ/p^e\ZZ}$. 
		\begin{mydef} Let~$A$ be a ring.
			Given $b:=\sum_{i\in\NN_0}b_it^{-i}\in A\llparenthesis t \rrparenthesis$ we define $$-\ord(b):=\sup (\{i\in\NN|\hspace{0,1cm}b_i\text{ is not nilpotent}\}).$$
			Note that, in particular $-\ord(b_0)=-\infty $ if $b_0\in A$.
		\end{mydef}
		\begin{mydef}Let $n\in\NN_0$ and let~$A$ be a ring. We denote by~$W^{\leq n}_e(A\llparenthesis t \rrparenthesis)$ the subset of $W_e(A\llparenthesis t \rrparenthesis)$ that has for objects the elements of form $b=(b_j)_{j=0}^{e-1}$ satisfying for $j=0\doots e-1$ that $$-\ord(b_j)\leq \frac{n-1}{p^{e-1-j}}.$$
			The elements of this form we call simple. If furthermore each $b_j$ is of the form $b_j=\sum_ib_{ij}t^{-i}$, we call them very simple elements and denote them by $(W^{\leq n}_e)^*(A\llparenthesis t \rrparenthesis)$.
		\end{mydef} 
		\begin{lem}\label{vsaddition} The simple and very simple elements are closed for the addition. 
		\end{lem}
		\begin{proof} Let 
			$$a:=\bigg(\sum_{i\in\NN\cup\{0\}}a_{ij}t^{-i}\bigg)_{j=0}^{e-1}\text{ and }b:= \bigg(\sum_{i\in\NN\cup\{0\}}b_{ij}t^{-i}\bigg)_{j=0}^{e-1}\in W^{\leq n}_e(A\llparenthesis t \rrparenthesis).$$ 
			For $j=0$, one has that the $0$-th coordinate of $a+b$ is given by $\sum_{i\in\NN\cup\{0\}}(a_{i0}+b_{i0})t^{-i}$  and clearly for every $i>(n-1)/p^{e-1-j}$, one has that $a_{i0}+b_{i0}=0$. Suppose for some $0\leq j\leq e-2$ one has for every $0\leq k\leq j-1$, that the $k$-th coordinate of~$a+b$ is of the form $\sum_{i\in\NN\cup\{0\}}c_{ij}t^{-i} $ with $c_{ij}=0$ whenever $i>(n-1)/p^{e-1-k}$. The $j$-th coordinate of $a+b$ is given by $$S_j\bigg(\bigg(\sum_{i\in\NN\cup\{0\}}a_{ij}t^{-i}, \sum_{i\in\NN\cup\{0\}}b_{ij}t^{-i}\bigg)_{k=0\doots j}\bigg).$$By Lemma~\ref{degwittaddition}, any monomial in~$S_j$ is of form $CX_0^{\alpha_0}Y_0^{\beta_0}\cdots X_j^{\alpha_j}Y_j^{\beta_j}$ for some $C\in\ZZ$ and some $\alpha_k,\beta_k$ satisfying that $\sum_{k=0}^jp^k(\alpha_k+\beta_k)=p^j$. Hence,
			\begin{align*}
				-\ord\bigg(\bigg(CX_0^{\alpha_0}Y_0^{\beta_0}\cdots X_j^{\alpha_j}Y_j^{\beta_j}\bigg)\bigg(\bigg(\sum_{i\in\NN\cup\{0\}}a_{ik}t^{-i}\bigg)_{k=0}^j,\bigg(\sum_{i\in\NN\cup\{0\}}b_{ik}t^{-i}\bigg)_{k=0}^j\bigg)\bigg)\hskip-11cm&\\
				&\leq\sum_{k=0}^{j}\alpha_k\cdot \bigg(-\ord\bigg(\sum_{i\in\NN\cup\{0\}}a_{ik}t^{-i}\bigg)\bigg) +\beta_k\cdot \bigg(-\ord\bigg(\sum_{i\in\NN\cup\{0\}}b_{ik}t^{-i}\bigg)\bigg)\\
				&\leq \sum_{k=0}^j (\alpha_k+\beta_k)\cdot \frac{n-1}{p^{e-k-1}}\\
				&=\sum_{k=0}^jp^k(\alpha_k+\beta_k)\frac{n-1}{p^{e-1}}\\
				&=\frac{(n-1)p^j}{p^{e-1}}\\
				&=\frac{n-1}{p^{e-j-1}}.
			\end{align*}
			It follows that for $j=0\doots e-1$, the minus order of $$S_{j}\bigg(\bigg(\sum_{i\in\NN\cup\{0\}}a_{ij}t^{-i}\bigg)_{k=0}^j,\bigg(\sum_{i\in\NN\cup\{0\}}b_{ij}t^{-i}\bigg)_{k=0}^j\bigg)$$ is thus no more than $(n-1)/p^{e-1-j}$ as claimed. It remains to verify that if~$a$ and~$b$ are very simple then their sum is such. But this follows because none of the polynomials $S_j$ has a constant term.
		\end{proof}
		\begin{mydef} Let $n\in\NN$. 
			Let us set $(\mathbb W^{\leq n}_{\Zpez})^*$ to denote the following $\FF_q$-algebraic group. As a scheme $$(\mathbb W_{\Zpez}^{\leq n})^*:=\prod_{j=0}^{e-1}\AAA^{\{i\in\NN|\hspace{0,1cm}i\leq (n-1)/p^{e-1-j}\}}.$$ 
			To obtain a group structure,   we identify the elements $$((a_{ij})^{\lfloor(n-1)/p^{e-1-j}\rfloor}_{i=1})^{e-1}_{j=0},((b_{ij})^{\lfloor(n-1)/p^{e-1-j}\rfloor}_{i=1})^{e-1}_{j=0}\in \mathbb  (W_{\Zpez}^{\leq n})^*$$ with the very simple elements $$\bigg(\sum_{i\in\NN}a_{ij}t^{-i}\bigg)_{j=0}^{e-1}, \bigg(\sum_{i\in\NN}b_{ij}t^{-i}\bigg)_{j=0}^{e-1}\in (W^{\leq n}_e)^*(A\llparenthesis t \rrparenthesis)$$ and add them there. As the very simple elements are closed for the addition, by reading the coefficients of the sum, we obtain the group law. We set $$\mathbb W_{\Zpez}^{\leq n}:=W_e\times (\mathbb W_{\Zpez}^{\leq n})^*$$which again is an algebraic group. We will often write $(\WW^{\leq n})^*$ and $\WW^{\leq n}$ for $(\WW^{\leq n}_{\Zpez})^*$ and $\WW^{\leq n}_{\Zpez}$, respectively.
		\end{mydef}
		By \cite[Theorem 16.59]{MilneAlgebraicGroups}, one has that $\WW^{\leq n}$ and $(\WW^{\leq n})^*$ are unipotent. For every $n\geq 1$ and every ring~$A$, the homomorphism $\wp_e(A\llparenthesis t \rrparenthesis):W_e(A\llparenthesis t \rrparenthesis)\to W_e(A\llparenthesis t \rrparenthesis)$ takes $W^{\leq n}_e(A\llparenthesis t \rrparenthesis)^*$ to~$W^{\leq n}_e(A\llparenthesis t \rrparenthesis)^*$ and is injective. 
		We deduce an injective homomorphism of algebraic groups  $$(\rho_{\Zpez}^n )^*:(\mathbb W^{\leq n})^*\to(\mathbb W^{\leq n+1})^*$$ for $n\geq 1$.  We obtain a homomorphism $$\rho_{\Zpez}^n=\wp_e\times (\rho^n)^*:\mathbb W^{\leq n}=W_e\times (\mathbb W^{\leq n})^*\to W_e\times (\mathbb W^{\leq n})^*=\mathbb W^{\leq n}.$$ 
		We set $$\rho_{\Zpez}^0:=\wp_e:W_e=\mathbb W^{\leq 1}\to W_e=\mathbb W^{\leq 1}.$$ 
		Again,  we may drop $\ZZ/p^e\ZZ$ from the index. One has a commutative diagram
		\[\begin{tikzcd}
			{(\mathbb W^{\leq n-2})^*} & {(\mathbb W^{\leq n-1})^*} \\
			{(\mathbb W^{\leq n-1})^*} & {(\mathbb W^{\leq n})^*.}
			\arrow["{(\rho^{n-1})^*}", from=1-1, to=1-2]
			\arrow[hook, from=1-1, to=2-1]
			\arrow[hook, from=1-2, to=2-2]
			\arrow["{(\rho^n)^*}"', from=2-1, to=2-2]
		\end{tikzcd}\]
		\begin{mydef}For $n\geq 2$, we denote by $\CZN$ the quotient algebraic group \cite[Theorem 5.14]{MilneAlgebraicGroups} 
			$$\CZN:=(\WW^{\leq n})^*/_{(\rho^{n-1})^*}(\WW^{\leq n-1})^*$$ which is necessarily commutative. 	We let $\CCC(\leq 1):=\Spec(\FF_q)$. We have a { canonical inclusion }$\CCC(\leq n-1)\hookrightarrow\CZN$.
		\end{mydef}
		By \cite[Corollary 14.7]{MilneAlgebraicGroups}, one has that $\CZN$ is a connected unipotent group. By \cite[Theorem 16.59]{MilneAlgebraicGroups}, we deduce $\CZN\cong \AAA^{\dim(\CZN)}$, where  
		$$\dim(\CZN)=\sum_{j=0}^{e-1}\bigg(\bigg\lfloor\frac{n-1}{p^{e-1-j}}\bigg\rfloor-\bigg\lfloor\frac{n-1}{p^{e-j}}\bigg\rfloor\bigg).$$
	Given $k\in\NN$, we define a morphism $\widetilde{\phi_k}: (\WW^{\leq n})^*\to\Delta_{\Zpez} $ by 
	$$\widetilde{\phi_k}:(a_{ij})_{i,j} \mapsto \bigg(\sum_{i\in\NN}a_{ij}t^{-ip^k}\bigg)_{j=0}^{e-1}.$$
	Set $\widetilde{\psi_k}:\WW^{\leq n}\to \Delta_{\Zpez}$ to be $(a_0,a)\mapsto a_0+\widetilde{\phi_k}(a).$
	\begin{lem}\label{psikinv} One has that $\widetilde{\psi_k}$ is a strict homomorphism of commutative group stacks. If $n\geq 2$, one has that $\widetilde{\psi_k}$ is invariant for the homomorphism $\rho^{n-1}:\WW^{\leq n-1}\to \WW^{\leq n}.$ If $n=1$, then $\widetilde{\psi_k}$ is $\rho^0$-invariant.
	\end{lem}
	\begin{proof}
		We verify that $\widetilde{\psi_k}$ is a strict homomorphism. Obviously, it suffices to verify that $\widetilde{\phi_k}$ is a strict homomorphism. The element $a=(a_{ij})_{i,j}$ will be identified with  simple element $(\sum a_{ij}t^{-i})_j$. Then for an element $b=(b_{ij})_{i_j}$ the sum of $a$ and $b$ is defined by reading the corresponding coefficients of $a+b$ when~$b$ is identified with $(\sum b_{ij}t^{-i})_j$. But one can replace $t^{-i}$ with $t^{-ip^k}$ and read the same coefficients. This means that $\widetilde{\psi_k}$ is indeed a strict homomorphism. Let $n\geq 2$. In order to verify that~$\widetilde{\psi_k}$ is $\rho^{n-1}$-invariant, for $u\in \WW^{\leq n-1}$, we will construct an isomorphism of $\widetilde{\psi_k}(x+\rho^{n-1}(u))$ and $\widetilde{\psi_k}(x)$. 
		Let $(\sum u_{ij}t^{-i})_j$ be the  simple element corresponding to~$u$ and $(\sum x_{ij}t^{-i})_j$ the one corresponding to~$x$. The element $\rho^{n-1}(u)$ corresponds to $\wp_e(u).$ 
		Let $(\sum c_{ij}t^{-i})_j:=(\sum x_{ij}t^{-i})_j+\wp_e(u)$. The fact that $$\wp_e\bigg(\bigg(\sum u_{ij}t^{-ip^k}\bigg)_j\bigg)=\bigg(\sum a_{ij}t^{-ip^k}\bigg)_j-\bigg(\sum x_{ij}t^{-ip^k}\bigg)_j$$ shows that $(\sum u_{ij}t^{-ip^k})_j$ is an isomorphism between $(\sum x_{ij}t^{-ip^k})_j=\widetilde{\psi_k}(x)$ and $(\sum c_{ij}t^{-ip^k})_j=\widetilde{\psi_k}(x+\rho^{n-1}(u))$. The proof when $n=1$ is established analogously.
	\end{proof}
	For a scheme~$X$, let us denote by $F_X$ the Frobenius morphism $F_X:X\to X$. It is immediate that the following diagram is commutative \[\begin{tikzcd}
		{W_e} & {W_e} \\
		{W_e} & {W_e.}
		\arrow["{F_{W_e}}", from=1-1, to=1-2]
		\arrow["{\wp_e}"', from=1-1, to=2-1]
		\arrow["{\wp_e}", from=1-2, to=2-2]
		\arrow["{F_{W_e}}"', from=2-1, to=2-2]
	\end{tikzcd}\] For $n\geq 2$ this gives the commutativity of the diagram:
	\[\begin{tikzcd}
		{(\mathbb W^{\leq n-1})^*} & {(\mathbb W^{\leq n})^*} \\
		{(\mathbb W^{\leq n-1})^*} & {(\mathbb W^{\leq n})^*}
		\arrow["{(\rho^{n-1})^*}", from=1-1, to=1-2]
		\arrow["{F_{(\mathbb W^{\leq n-1})^*}}"', from=1-1, to=2-1]
		\arrow["{F_{(\mathbb W^{\leq n})^*}}", from=1-2, to=2-2]
		\arrow["{(\rho^{n-1})^*}"', from=2-1, to=2-2]
	\end{tikzcd}\]
	and the vertical maps are homomorphisms of algebraic groups.
	Lemma~\ref{psikinv} and the fact that $(\rho^{n-1})^*=\rho^{n-1}|_{\{0\}\times (\WW^{\leq n-1})^*}$ provide a homomorphism of group stacks $${\phi_k}:\CZN\to \Delta_G $$ $${\psi_k}:\WW^{\leq n}/_{\rho^{n-1}}\WW^{\leq n-1}=(W_e/\wp_eW_e)\times \CZN \to \Delta_G$$ for $n\geq 1$ and $k\geq 1$. One has that $W_e/\wp_eW_e$ is none other than $B(\ZZ/p^e\ZZ)$. Moreover, one has an arrow
	$${\psi_{k+1}}\circ (\Id_{B(\ZZ/p^e\ZZ)}\times F_{\CZN})(a_0,a)\xrightarrow{-{\phi_k}(a)}\psi_k(a_0,a)$$
	which provides a homomorphism of commutative group stacks $$\psi:B(\ZZ/p^e\ZZ)\times \CZN^{\infty}\to \Delta_G,$$
	where $\CZN^{\infty}$ stands for the ind-perfection. Let us write $\psi^n$ of $\phi^n$ when we want to stress the dependence on~$n$.  One has an inductive system of categories fibered in $2$-groups $$(B(\Zpez)\times \CCC(\leq 1)^{\infty}\to B(\Zpez)\times \CCC(\leq 2)^{\infty} \to\cdots$$ such that the transition maps commute with $\psi^n$. 
	\begin{thm}\label{cycequiv}
		One has an equivalence \begin{align*}\psi:=\varinjlim_n\psi^n:\varinjlim_n\big(B(\Zpez)\times\CZN^{\infty}\big)&=B(\Zpez)\times \varinjlim_n\big(\CZN^{\infty}\big) \\&\xrightarrow{\sim}\Delta_{\Zpez}
		\end{align*}	of categories fibered in $2$-groups.
	\end{thm}
	\begin{proof} We need to verify that $\varinjlim_n\psi^n$ is essentially surjective and fully faithful.
		
		{\it Essential surjectivity.} Let $$a:=(a_j)_{j}=\bigg(\sum_{i\in\NN\cup\{0\}}a_{ij}t^{-i}\bigg)_j\in\Delta_{\Zpez}.$$ First, note that for $j=0\doots e-1$, one can inductively choose elements~$c_j$ such that the free term of the entries $0\doots j$ of $$a-(c_0,0\doots 0)-(0,c_1\doots ,0)-\cdots (0\doots c_j, 0\doots 0)$$are~$0$. Indeed, choose $c_0=-a_{00}$, then the free term of the $0$-th entry of $a-(c_0,0\doots 0)$ is $0$, then choose $c_1$ so that the free term of $1$-st entry of $a-(c_0,0\doots 0)-(0,c_1,0\doots 0)$ is~$0$, etc.\ Thus, up to replacing~$a$ by $a-(c_0\doots c_{e-1})$ we can assume that~$a$ is very simple. We now show that~$a$ is isomorphic to an element of form $b=(\sum_{i\in\NN'} b_{ij}t^{-ir_j})_j$ with $r_j$ powers of~$p$. We do it by induction on $j=0\doots e-1$. Suppose $j=0$. If $0$-th coordinate of~$b$ is~$0$ we are done, and if not, set $$r_0:=p^{\max_{a_{i0}\neq 0}v_p(a_i)}.$$ Set $i':=i/p^{v_p(i)}$. The element:\begin{align*}w(a_0):&=\sum_{a_{i0}\neq 0}(a_{i0})^{p^{r_0-v_p(i)}}t^{-i'p^{r_0}}-(a_{i0})t^{-i}\\
			&=\sum_{a_{i0}\neq 0} \sum_{k=0}^{r_0-v_p(i)-1}(a_{i0})^{p^{k+1}}t^{-i'p^{k+1}}-(a_{i0})^{p^{k}}t^{-ip^k}
		\end{align*}is a sum of the elements of the form $u^p-u$, hence itself is of the form $u^p-u$. We have that $a+(w(a_0),0\doots 0)=(a_0'\doots a_{e-1}')$ with $a_0'$ of the wanted form and $(a_j')_j\simeq (a_j)_j$. Now we add $(0,w(a_1'),0\doots ,0)$ and obtain an element where the $0$-th and $1$-st coordinate are of the wanted form. We continue inductively to reach an element of the wanted form. Now suppose that $b=\sum_{i\in\NN'}b_{ij}t^{-r_j}=:(b_j)_j$ with $r_j$ powers of~$p$. Let us establish that we can choose $r_j$ the same. Indeed, let $r $ be the largest among $r_j$. One has  $$b=(b_j)_j:=\sum_{j=0}^{e-1}((b_j)_{k=j}, 0\doots 0).$$ We have that $$((b_j)_{k=j}, (0)_{k\neq j})\simeq ((b_j^p)_{k=j},(0)_{k\neq j})\simeq\cdots \simeq ((b_j^{p^{r-r_j}})_{k=j}, (0)_{k\neq j}).$$ Here~$p$ in exponent means taking $p$-th power and not taking element at this index. The claim is established. To prove essential surjectivity, it suffices to observe that $$(\sum _{i\in\NN'}b^{p^r}_{ij}t^{-ip^r})_j$$ is the image  of the element $$\widehat b:=(b^{p^r}_{ij})_{i,j}\in \WW^{\leq n},$$where $$n\geq 1+\max_{j=0\doots e-1} (\max_{i: b^{ip^r}_j\neq 0} p^{e-j-1}i),$$ under  the map $\widetilde{\phi_r}.$ 
		
		{\it Faithfulness.}  Consider two elements $\overline{(b_0,b)}$ and $\overline{c_0,c}$ of $\varinjlim_n(B(\ZZ/p^e\ZZ)\times\CZN)$ and two morphisms $\overline{u},\overline{v}$ between them. Suppose $\psi(\overline u)=\psi(\overline v)$ Let~$n$ be sufficiently big so that there exists $(b_0,b),(c_0,c)\in B(\Zpez)\times\CZN^{\infty}$ and $u,v\in \Hom(b,c)$  mapping to the two elements and two morphisms, respectively.  Let $k$ be such that $(b_0,b),(c_0,c),u,v$ are in the image of $k$-th copy of $B(\Zpez)\times\CZN$ for the canonical map $B(\Zpez)\times\CZN\to B(\Zpez)\times\varinjlim_n\CZN.$ Let $\widetilde x$ denotes a preimage of~$x$ for $x\in\{b,c,u,v\}$. But $\CZN$ is a scheme, thus $\widetilde b=\widetilde c$. We also deduce $b_0+\wp_e(\widetilde u)-\widetilde u=c_0$ and $b_0+\wp_e(\widetilde v)-\widetilde v=c_0$. In particular, one has that $u,v$ are constants. Now $\psi(\overline u)=\psi (\overline v)$ implies $\psi_k(u)=\psi_k(v)$ which implies $u=v$. The faithfulness follows.
		
		{\it Fullness.} Take two elements in~$\Delta_{\Zpez}$. Up to changing them by isomorphic elements, they are necessarily of the form $\psi_k(x),\psi_k(y)$ with $x,y\in (B(\Zpez)\times \CZN)$ for some $k\geq 1$ and some $n\geq 1$. The set of isomorphisms is $\Hom(\psi_k(x),\psi_k(y))=\{u|\hspace{0,1cm}\psi_k(x)+\wp_e(u)=\psi_k(y)\}$. Thanks to Lemma~\ref{wedelta}, one has $u=(\sum_{i\geq 0} u_{ij}t^{-i})_j$ for certain $u_{ij}$. We prove that whenever $u^i_j\neq 0$, then $v_p(i)\geq k$. Suppose there exists $(i,j)$ such that $u^i_j\neq 0$ and $v_p(i)<k$. Choose the smallest possible~$j$ and then choose~$i$ with the smallest value of $v_p(i)<k$. The $j$-th entry of $\wp_e(u)$ is equal to $$S_j(u^p_0\doots u^p_j, I_0(u_0)\doots I_j(u_0\doots u_j)),$$where $p$ in the exponent means taking $p$-th power, while $I_k$ is the $k$-th polynomial for the additive inverse. By Lemma~\ref{wittinverse}, one has that $I_j$ is linear in~$X_j$. Moreover, by Lemma~\ref{degwittaddition}, one has that~$S_j$ is linear in~$X_j$ and~$Y_j$. Thus the formula for the $j$-th entry of $\wp_e$ is linear in~$X_j$ and~$Y_j$. By using the minimality of $v_p(i)$, we have that the coefficient in front of~$t^{-i}$ in the $j$-th coordinate of $\wp_e(u)$ is $-u^i_j$. Finally, by using that $\psi_k(x)$ is of the form $(\sum_{i'\in\NN'}d_{ij}t^{-i'p^k})_j$ and by using that the formula for~$j$-th term of $S_j$ is linear in $X_j$ and $Y_j$, hence, the $j$-th term of $\psi_k(x)+u$ has the coefficient in front of~$t^{-i}$ also $u_{ij}$, in particular non-zero. We deduce a contradiction, as $\psi_k(y)$ is also of form $(\sum_{i'\in\NN'} c^{}_{i'j}t^{-{i'}p^k})_j$ and $v_p(i)<k$.  Hence, one has that $u=(\sum_{i\geq 0}u_{ij}t^{-i})$ with the condition that $u^i_j=0$ whenever $v_p(i)<k$. Then, one has that $u$ is the image of $(u_{ij})_{i,j}\in\WW^{\leq n-1}$. The proof is completed.
	\end{proof}
	\subsubsection{}Let us now treat the case of a general finite $p$-group $G=\prod_{u=1}^d\ZZ/p^{e_u}\ZZ$. We set $$(\WW^{\leq n}_G)^*:=\prod_{u=1}^d(\WW^{\leq n}_{\ZZ/p^{e_u}\ZZ})^*$$ and $$\WW^{\leq n}_G=W_G\times (\WW^{\leq n}_G)^*.$$
	We have homomorphisms $$\rho^n_{G}=\prod_{u=1}^e\rho^{n}_{\ZZ/p^{e_u}\ZZ}:(\WW^{\leq n}_G)\to (\WW^{\leq n+1}_G)^*.$$
	We set \begin{align*}\CCC_G(\leq n):=(\WW^{\leq n+1}_G)^*/_{\rho^n_G}\WW^{\leq n}_G=\prod_{u=1}^d\CCC_{\ZZ/p^{e_u}\ZZ}(\leq n).
	\end{align*}
	We also set $\CCC(\leq 1):=\Spec(\FF_q)$.  By setting~$\psi_G$ to be the product of the corresponding maps~$\psi$ for each~$u$, we obtain the following corollary. 
	\begin{cor}\label{strdeltag}
		One has that
		$$\psi_G: BG\times\varinjlim_n \CCC_G(\leq n)^{\infty}\xrightarrow{\sim}\Delta_G$$
		is an equivalence of categories fibered in $2$-groups.
	\end{cor}
	\subsubsection{} Let us prove a functoriality. Given a homomorphism of finite abelian $p$-groups $f:G\to H$, we have an induced homomorphism of categories fibered in $2$-groups $\Delta_f:\Delta_G\to\Delta_H$. We will write~$\Delta_f$ as a limit of certain functors. 
	Given $n\geq 1$, let~$W_n$ or $W_{\ZZ/p^n\ZZ}$ denote the group scheme of Witt vectors of length~$n$ \cite[Page 215]{MotivicIntegration}. When $n=0$, one should understand~$W_0$ to be the trivial group scheme. Set $\wp_n:W_n\to W_n$ for the homomorphism $x\mapsto F(x)x^{-1}$, where $F(x)$ denotes the Frobenius homomorphism. One has that $\ker (\wp_n)=W_n(\FF_p).$ By Lang theorem \cite[Corollary of Theorem 1]{Langtheorem} the sequence
	$$1\to \ZZ/p^n\ZZ=W_n(\FF_p) \to W_n\xrightarrow{\wp_n}  W_n\to 1$$
	is exact. 
	For a finite abelian $p$-group $G=\prod_{u=1}^d (\ZZ/p^{e_u}\ZZ),$ we set $$W_G:=\prod_{i=1}^dW_{\ZZ/p^{e_u}\ZZ}.$$ We deduce an exact sequence:
	$$1\to G=W_G(\FF_p)\to W_G\xrightarrow{\wp_G}W_G\to 1,$$
	where $\wp_G$ is the product of morphisms $\wp_{e_u}$. 
	In this subsection, for any homomorphism~$\phi$ of finite abelian $p$-groups, we construct a homomorphism~$W_{\phi}$ of certain group schemes of Witt vectors such that $W_{\phi}(\FF_p)=\phi$. 
	\subsubsection{}\label{cyclicwphi} 
	Let $n\geq 1$ be an integer and let $G=\ZZ/p^n\ZZ$. 
	Whenever $n\geq k\geq 1$, we have a canonical surjection $$s_{n,k}:\ZZ/p^{n}\ZZ\twoheadrightarrow \ZZ/p^k\ZZ,\hspace{1cm}x+p^n\ZZ\mapsto x+p^k\ZZ.$$Write~$W_{s_{n,k}}$ for the surjection of Witt group schemes $$W_{s_{n,k}}:W_{n}\twoheadrightarrow W_k\hspace{1cm}(a_j)_{j=0}^{n-1}\mapsto (a_{j})_{j=0}^{k-1}.$$ Whenever $n\geq k\geq 1$, we have a canonical injection $$i_{k,n}:\ZZ/p^k\ZZ\hookrightarrow\ZZ/p^n\ZZ,\hspace{1cm}x+p^k\ZZ\mapsto p^{n-k}x+p^n\ZZ.$$
	One has that $W_{s_{n,k}}(\FF_p)=s_{n,k}.$ 
	Write~$W_{i_{k,n}}$ for the inclusion of Witt group schemes $$W_{i_{k,n}}:W_{k}\hookrightarrow W_{n},\hspace{1cm} (a_j)_{j=0}^{k-1}\mapsto ((0)_{j=0}^{n-k-1},(a_{j-(n-k)})_{j=n-k}^{n-1})$$ 
	One has that $W_{i_{k,n}}(\FF_p)=i_{k,n}$. Moreover, if $n=k$, then $i_{n,n}=s_{n,n}=\Id_{\ZZ/p^n\ZZ}$ and $W_{i_{n,n}}=W_{s_{n,n}}=\Id_{W_{n}}$. 
	\begin{mydef}
		Let $G=\prod_{u=1}^d\ZZ/p^{e_u}\ZZ$ be a finite abelian $p$-group. We define a type of a homomorphism $\phi:G\to H$, with $H$ finite abelian $p$-group.
		\begin{enumerate}
			\item If $H=\ZZ/p^h\ZZ$ is cyclic, then the type of $\phi:G\to H$ is given by $d$-tuple of pairs $$\type(\phi):=\big(\phi((0)_{i\neq u}, (\overline{1})_{i=u}), \rho_{e_u,h}\big)_{u=1}^d,$$ where $\overline 1$ denotes the image of~$1$ in~$\ZZ/p^{e_u}\ZZ,$ and
			$$\rho_{e_u,h}:=\begin{cases}s_{e_u,h},\text{ if $e_u\geq h$}\\i_{e_u,h}, \text{otherwise}  \end{cases}.$$
			\item More generally, if $H=\prod_{h=1}^r(\ZZ/p^{g_h}\ZZ),$ for a homomorphism $\phi:G\to H$, let $\phi_h:G\to \ZZ/p^{g_h}\ZZ$ be such that $\phi=\prod_{h=1}^r\phi_h$. We define $$\type(\phi):=(\type(\phi_h))_{h=1}^r.$$
		\end{enumerate}
	\end{mydef}
	It is immediate from the definition that $\type$ is well-defined for any homomorphism of finite abelian $p$-groups.
	For an element $k\in\ZZ/p^h\ZZ$, we denote by $\widetilde k\in[0,\doots p^h-1]$ the element the image of which under $\ZZ\to\ZZ/p^h\ZZ$ is~$k$. For $m,n\in\NN$, we have a homomorphisms of group schemes $$m\times-:W_n\to W_n,\hspace{1cm}x\mapsto mx$$ and $$+^m:W_n^m\to W_n^m,\hspace{1cm}(x_i)_{i=1}^n\mapsto \sum_{i=1}^nx_i.$$
	We may write $m\times-$ and $+^m$ for $(m\times-)(\FF_p)$ and $+^m(\FF_p)$, respectively.
	\begin{mydef}
		Given a homomorphism $$\phi:G=\prod_{u=1}^d\ZZ/p^{e_u}\ZZ\to H=\prod_{h=1}^r\ZZ/p^{g_h}\ZZ$$ of the type $$((k_u, \rho_{e_u,g_h})_{u=1}^d)_{h=1}^r ,$$  we define $W_\phi:W_G\to W_H$ by $$W_{\phi}:=\prod_{h=1}^r+^{d}\circ ((({\widetilde{k_u}\times-})\circ W_{\rho_{e_u,g_h}})_{u=1}^d)$$
	\end{mydef}
	\begin{lem} \label{phighw} For any homomorphism of finite abelian $p$-groups $\phi:G\to H$, one has that  $$W_{\phi}(\FF_p)=\phi.$$ 
		Moreover, one has that
		\[\begin{tikzcd}
			0 & G & {W_G} & {W_G} & 0 \\
			0 & H & {W_H} & {W_H} & 0
			\arrow[from=1-1, to=1-2]
			\arrow[from=1-2, to=1-3]
			\arrow["\phi"', from=1-2, to=2-2]
			\arrow["{\wp_G}", from=1-3, to=1-4]
			\arrow["{W_{\phi}}"', from=1-3, to=2-3]
			\arrow[from=1-4, to=1-5]
			\arrow["{W_{\phi}}"', from=1-4, to=2-4]
			\arrow[from=2-1, to=2-2]
			\arrow[from=2-2, to=2-3]
			\arrow["{\wp_H}", from=2-3, to=2-4]
			\arrow[from=2-4, to=2-5]
		\end{tikzcd}\]
		is commutative.
	\end{lem}
	\begin{proof}
		Recall that $W_{\rho_{e_u,g_h}}$ are defined so that $W_{\rho_{e_u,g_h}}(\FF_p)=\rho_{e_u,g_h}.$
		The maps $\widetilde{k_u}\times -$ and $+^d $ clearly satisfy that $(\widetilde{k_u}\times -)(\FF_p)$ is the multiplication by $\widetilde k_u$ in $W_{g_h}(\FF_p)$ and the addition in $W_{g_h}(\FF_p)^d.$ By taking the products we obtain $W_{\phi}(\FF_p)=\phi$. We now verify the commutativity of the second square, that is $W_{\phi}\circ \wp_G=\wp_H\circ W_{\phi}$. But the map $\wp_G$ is the multiplication  of the Frobenius and the inverse map, hence commutes for a homomorphism of algebraic groups.
	\end{proof}
	By \cite[Section 1.4.2]{Dualite}, we obtain that:
	\begin{cor} Let $\phi:G\to H$ be a homomorphism of finite abelian $p$-groups. The homomorphism $\Delta_{\phi}:\Delta_G\to \Delta_H$ of categories fibered in $2$-groups is isomorphic to the homomorphism $$\Delta_{G}\simeq W_G/\wp_GW_G\to W_H/\wp_H W_H\simeq  \Delta_H$$ which is induced from the diagram in Lemma~\ref{phighw}.
	\end{cor}
	\subsubsection{}
	\begin{lem} For an integer~$m$, we set $$W^{\leq m}_G(A\llparenthesis t \rrparenthesis):=\prod_{u=1}^d W^{\leq m}_{e_i}(A\llparenthesis t \rrparenthesis).$$ Let $\phi:G\to H$ be  homomorphism of finite abelian $p$-groups. Let $m\geq 1$ be an integer. One has that $$W_{\phi}\big(W_G^{\leq m}(A\llparenthesis t \rrparenthesis)\big)\subset W_H^{\leq m}(A\llparenthesis t \rrparenthesis).$$
		One also has that $$W_{\phi}\big(W_G^{\leq m}(A\llparenthesis t \rrparenthesis)^*\big)\subset W_H^{\leq m}(A\llparenthesis t \rrparenthesis)^*.$$
	\end{lem}
	\begin{proof}
		Obviously, it suffices to treat the case when~$H=\ZZ/p^{h}\ZZ$ is cyclic. Let $(k_u,\rho_{e_u,h})_{u=1}^d$ be the type of~$\phi$. By using the obvious property that $W_{\rho_{e_u,h}}$ takes $W_{e_u}^{\leq m}(A\llparenthesis t \rrparenthesis)$ to $W^{\leq m}_h(A\llparenthesis t \rrparenthesis)$ and by using the two properties implied by Lemma~\ref{vsaddition} that $(\widetilde{k_u}\times -)$ takes $W_{h}^{\leq m}(A\llparenthesis t \rrparenthesis)$ to $W^{\leq m}_h(A\llparenthesis t \rrparenthesis)$ and that $+^{d}$ takes $W_{h}^{\leq m}(A\llparenthesis t \rrparenthesis)^d$ to $W_{h}^{\leq m}(A\llparenthesis t \rrparenthesis)$, we obtain that:
		\begin{align*}
			W_{\phi}\big(W^{\leq m}_{e_u}(A\llparenthesis t \rrparenthesis)\big)&=\big(+^{d}\circ\big(((\widetilde{k_u}\times -)\circ(W_{\rho_{e_u},h}))_{u=1}^d\big)\big) (W^{\leq m}_{e_u}(A\llparenthesis t \rrparenthesis))\\
			&\subset +^{d} \big((\widetilde{k_u}\times-)(W^{\leq m}_h(A\llparenthesis t \rrparenthesis))_{u=1}^d\big)\\
			&\subset +^{d}\big( W_h^{\leq m}(A\llparenthesis t \rrparenthesis)^d\big)\\
			&\subset W_h^{\leq m}(A\llparenthesis t \rrparenthesis).
		\end{align*} 
		Let us prove the analogous claim for $W_G^{\leq m}(A\llparenthesis t \rrparenthesis)^*$. Then clearly, one has that $W_{\rho_{e_u},h}$ takes $W_{e_u}^{\leq m}(A\llparenthesis t \rrparenthesis)^*$ to $W^{\leq m}_h(A\llparenthesis t \rrparenthesis)^*$.  One has that $\widetilde {k_u}\times -$ takes $W_{h}^{\leq m}(A\llparenthesis t \rrparenthesis)^*$ to $W_h^{\leq m}(A\llparenthesis t \rrparenthesis)^*$, because multiplication by $\widetilde {k_u}\times -$ is given by a polynomial without free coefficient. Similarly, the addition in $W_h^d$ is given by a polynomial without free coefficient, hence $+^d$ takes $((W_h^{\leq m}(A\llparenthesis t \rrparenthesis))^*)^d$ to $W_h^{\leq m}(A\llparenthesis t \rrparenthesis)^*$. The statement follows.
	\end{proof}
	Recall that for a ring~$A$, one has  $$(\WW_G^{\leq m})^*(A)=W_G^{\leq m}(A\llparenthesis t \rrparenthesis)^*\text { and }(\WW_H^{\leq m})^*(A)=W_H^{\leq m}(A\llparenthesis t \rrparenthesis)^*.$$
	From a homomorphism $W_\phi^{\leq m}(A\llparenthesis t \rrparenthesis):W^{\leq m}_G(A\llparenthesis t \rrparenthesis)^*\to W^{\leq m}_H(A\llparenthesis t \rrparenthesis)^*$, we deduce a group functor $\WW_{\phi}^{\leq m}: (\WW^{\leq m}_G)^*\to (\WW^{\leq m}_H)^*$ satisfying $\WW_{\phi}^{\leq m}(A)=W_\phi^{\leq m}(A\llparenthesis t \rrparenthesis)$, which, by the Yoneda lemma, is necessarily a morphism of group schemes. Moreover, if $G=\prod_{u=1}^d\ZZ/p^{e_u}\ZZ$, $H=\prod_{h=1}^g\ZZ/p^{e_h}\ZZ$ and  $$\type(\phi)=\prod_{h=1}^g(+^d\circ ((k_u\times -)\circ \rho_{e_u,e_h})_{u=1}^d),$$then 
	$$\WW_{\phi}^{\leq n}=\prod_{h=1}^g(+^d\circ((\widetilde k_u\times -)\circ \WW_{\rho_{e_u,e_h}})_{u=1}^d),$$where as usual $\widetilde{k_u}\in [0\doots p^{e_u}-1]$ is the lift of~$k_u\in\ZZ/p^{e_u}\ZZ$.  
	\subsubsection{} For any finite abelian $p$-group $G=\prod_{u=1}^d(\ZZ/p^{e_u}\ZZ)$ and any element $a\in(\WW_G^{\leq m})^*(\overline{\FF_q})$, where $m\geq 0$, we define a ``path'' from $0$ to~$a$. 
	Write $a=(a_u)_u,$ with $$a_u=(a_{uj})_{j=0}^{e_u}\in\prod_{j=0}^{e_u}\AAA^{\{x\in\NN|\hspace{0,1cm}x\leq (m-1)/p^{e-j}\}}(\overline{\FF_q}).$$
	We define an $\overline{\FF_q}$-morphism:$$E_{G,a}:\AAA^1\to(\WW_G^{\leq m})^*,\hspace{1cm} t\mapsto (t^{p^j}a_{uj})_{j=0}^{e_u}.$$
	We set $\eta_{G,a}:E_{G,a}(\overline{\FF_q})$ to be the image of this morphism on $\overline{\FF_q}$-points. When one identifies $(\WW_G^{\leq m})^*$ with closed subvariety of $(\WW_G^{\leq m+1})^*$ by setting the latter coordinates to be~$0$, it is clear that~$E_{G,a}$ and $\eta_{G,a}$ do not depend on the choice of~$m$, as the notation suggests.
	%
	\begin{lem}\label{pathsave} Let $m\geq 0$. Consider a homomorphism of finite abelian $p$-groups $\phi:G\to H$.  Let $a\in (\WW_G^{\leq m})^*(\overline{\FF_q}).$  One has that $$(\WW_{\phi}^{\leq m})^*(\eta_{G,a})=\eta_{H,(\WW_{\phi}^{\leq m})^*(a)}.$$
	\end{lem}
	\begin{proof}
		If the property is satisfied for the two composable morphisms~$f$ and~$g$, then it is clearly satisfied for~$g\circ f$. By the definition of~$(\WW_{\phi}^{\leq m})^*$, it suffices thus to verify it in the following five cases. 
		\begin{enumerate}
			\item The case $G=\ZZ/p^e\ZZ$, $H=\ZZ/p^h\ZZ$, $e\geq h$ and $\phi=s_{e,h}$. Then the image of an element $t$ for the composite $(\WW_{\phi}^{\leq m})^*\circ E_{G,a}$ is $$(\WW_{\phi}^{\leq m})^{*}(E_{G,a}(t)) =(t^{p^j}a_{j})_{j=0}^{h-1},$$which is precisely the image of~$t$ for $E_{H,(\WW_{\phi}^{\leq m})^*(a)}.$ We deduce that in this case $(\WW_{\phi}^{\leq m})^*\circ E_{G,a}=E_{H,(\WW_{\phi}^{\leq m})^*(a)}$ which is stronger than the wanted claim. 
			\item The case $G=\ZZ/p^e\ZZ$ and $H=\ZZ/p^h\ZZ$ and $e< h$ and $\phi=i_{e,h}$. Then the image of an element $t$ for the composite $(\WW_{\phi}^{\leq m})^*\circ E_{H,a}$ is $$(\WW^{\leq m}_{\phi})^{*}(E_{G,a}(t)) =((0)^{h-e-1}_{j=0}),(t^{p^{j-(h-e)}}a_{j-(h-e)})_{j=h-e}^{h-1}.$$ On the other hand, the image of~$t$ for $E_{H,(\WW_{\phi}^{\leq m})^*a}$ is $$E_{H,(\WW_{\phi}^{\leq m})^*a}(t)=((0)_{j=0}^{h-e-1}, (t^{p^j}a_{j-(h-e)})_{j=h-e}^{j=h-1}).$$
			We deduce $$\Wphim(\eta_{G,a})=\eta_{H,\Wphim(a)} $$.
			\item Suppose that $G=(\ZZ/p^h\ZZ)^d$, for certain $d\geq 1$, while  $H=\ZZ/p^h\ZZ$ and $\phi=+^d$ is the addition in~$G$. The element $a$ corresponds to $$\bigg(\bigg(\sum_{i\in\NN}a_{uij}t^{-i}\bigg)_j\bigg)_{u}\in W_{G}(\overline{\FF_q}\llparenthesis t \rrparenthesis).$$The element $\Wphim (a)$ corresponds to $\sum _u a_{u}$. The $j$-th coordinate of the latter element is, by Lemma~\ref{wittbig}, given by a polynomial in $a_{uij}$ and its degree is $p^j$ when $a_{uij}$ is given weight $p^j$.  We deduce that $$ (\Wphim)\circ E_{G,a} =E_{H,\Wphim(a)}.$$
			\item Suppose that $G=H=\ZZ/p^h\ZZ$ and $\phi=k\times-$ is the multiplication by $k\in\ZZ/p^h\ZZ$. By the definition, one has $(\WW_{\phi}^{\leq m})$ is multiplication by $\widetilde k$ in $(\WW_{\phi}^{\leq m})$, with $\widetilde k\in[0\doots p^{e_u}-1].$ This corresponds to the multiplication by~$\widetilde k$ in $W_G^{\leq m}(\overline{\FF_q}\llparenthesis t \rrparenthesis)$. It follows from Lemma~\ref{wittbig}, that $$(\Wphim)\circ E_{G,a} =E_{H,\Wphim(a)},$$ which is stronger that our claim.
			\item Suppose that $\phi=\phi_1\times \phi_2:G=G_1\times G_2\to H_1\times H_2$, with $G_1, G_2, H_1,H_2$ finite abelian $p$-groups and $\phi_1,\phi_2$ homomorphisms of finite abelian $p$-groups. Then $\Wphim= (\WW_{\phi_1}^{\leq m})^*\times (\WW_{\phi_2}^{\leq m})^*$ and clearly $E_{G, (a_1,a_2)}= E_{G,a_1}\times E_{G,a_2}.$ By using induction, we can conclude that the property in the question is satisfied.
		\end{enumerate}
	\end{proof}
	\subsubsection{}The commutativity in the diagram in Lemma~\ref{phighw} gives the commutativity of 
	\[\begin{tikzcd}
		{(\WW^{\leq m}_G)^*} & {(\WW^{\leq m+1}_G)^*} \\
		{(\WW_H^{\leq m})^*} & {(\WW_H^{\leq m+1})^*}
		\arrow["{(\rho_G^m)*}", from=1-1, to=1-2]
		\arrow["{\WW_{\phi}^{\leq m}}"', from=1-1, to=2-1]
		\arrow["{\WW_{\phi}^{\leq m}}"', from=1-2, to=2-2]
		\arrow["{(\rho_H^{m})^*}", from=2-1, to=2-2]
	\end{tikzcd}\]
	for $m\geq 1$. We deduce a homomorphism of commutative algebraic groups: $$\CCC_{\phi}(\leq m):\CCC_G(\leq m)\to\CCC_H(\leq m)$$
	and consequently a morphism of sheaves
	$$\CCC_{\phi}(\leq m)^{\infty}:\CCC_G(\leq m)^{\infty}\to\CCC_H(\leq m)^{\infty}.$$
	\begin{thm}\label{comdeltag}
		For any $m\geq 1$, one has that the restriction of $\Delta_{\phi}$ to the subcategory $BG\times \CCC_{G}(\leq m)^{\infty} $ is isomorphic to the homomorphism $B\phi\times \CCC_{\phi}(\leq m)^{\infty}$ of categories fibered in~$2$-groups.
	\end{thm}
	\begin{proof}It suffices to treat the case $H=\ZZ/p^h\ZZ$ is cyclic. We take an element of the form $(b_0,0)\in BG\times \CCC_G(\leq m)^{\infty}$. Its image in $\Delta_G$ is $b_0$. The image of~$b_0$ for $\Delta_{\phi}$ is the class of $(B(\phi))(b_0)$. This is the image of $((B(\phi))(b_0),0)$ for the canonical morphism $BH\times \CCC_{H}(\leq m)^{\infty}$. We now take an element of the form $(0,b)\in BG\times \CCC_G(\leq m)^{\infty}(A)$, where~$A$ is a ring. For some $k\geq 1$, one has that $(0,b)$ is isomorphic to the image of some $(0,b)\in BG\times\CCC_G(\leq m)$ under the canonical map from the~$k$-th copy of $\CCC_G(\leq m)$ in $\CCC_G(\leq m)\to\cdots$ to $\CCC_G(\leq m)$. To differentiate between the~$\psi$ map for~$G$ and~$H$, we write $\psi^G$ and $\psi ^H$. The image of $(0,b)$ in~$\Delta_G$ is then isomorphic to~$\psi^G_k(0,b_k)$.  By Lemma~\ref{awnz}, the element $b_k$ lifts to an element $\widetilde{b}_k\in (\WW^{\leq m}_G)^*(A')$. One has that $\psi^G_k(0,b_k)\simeq \widetilde{\psi}^G_k(0,\widetilde b_k)$. If we write $\widetilde{b}_k=(b_{i,j,u,k})_{i,j}$
		then $$\widetilde\psi^G_k(0,\widetilde b_k)=\bigg(\sum_{i\in\NN'}b_{i,j,u}t^{-ip^k}\bigg)_{j,u}.$$ 
		If the type of $\phi$ is $(k_u,\rho_{e_u,h})_{u=1}^d,$ the image of the latter element under~$\Delta_{\phi}$ is isomorphic to \begin{align*}
			\overline b:=	\bigg(0,\bigg(+^d\circ \bigg( (\widetilde{k_u}\times -)\circ W_{\rho_{e_u,h}} \bigg)\bigg)\bigg(\bigg(\sum_{i\in\NN}b_{i,j,u}t^{-ip^k}\bigg)_{j,u}\bigg)\bigg).
		\end{align*}
		On the other side, the element $(0,b)$ maps  under $B\phi\times \CCC_G(\leq m)^*$ to an element isomorphic to~$b'$ corresponding to \begin{align*}(0,\WW_{\phi}(\widetilde {b}_k))= \bigg(0,\bigg(+^d\circ \bigg( (\widetilde{k_u}\times -)\circ W_{\rho_{e_u,h}} \bigg)\bigg)\bigg(\bigg(\sum_{i\in\NN}b_{i,j,u}t^{-i}\bigg)_{j,u}\bigg)\bigg)
		\end{align*}
		with the standard identification of  $(\WW_H^{\leq m})^*(A)=W_H^{\leq m}(A\llparenthesis t \rrparenthesis)^*$. We need to verify that $\widetilde \psi_k(b')\simeq \overline b$. But this follows from the fact $\widetilde\psi_k((b_{i,j,u})_{j,u})=(\sum b_{i,j,u}t^{-ip^k})_{j,u}$ and the fact that~$\widetilde\psi_k$ is a strict homomorphism.
	\end{proof}
	\subsubsection{} 
		Let us denote $\pi(\leq n):(\WW_G^{\leq n})^*\to \CCC_G(\leq n)$ the quotient map. For $a\in \CCC_G(\leq n)(\overline {\FF_q})$, we define $$[\eta_{G,a}]:=\pi(\leq n)(\eta_{G,\widetilde a}),$$where $\widetilde a$ is any lift of $a$. 
		\begin{lem}\label{quotpath}
			Given a homomorphism $\phi:G\to H$, for any $a\in (\WW_G^{\leq m})^*(\overline {\FF_q})$, one has that $$\CCC_{\phi}({\leq m})([\eta_{G,a}]) =[\eta_{H, \CCC_G(\leq m)(a)}].$$
		\end{lem}
		\begin{proof}
			This is immediate by Lemma~\ref{pathsave} which gives that $(\WW_G^{\leq m})^*(\eta_{G,\widetilde a}) =\eta_{H,(\WW_H^{\leq m})^*(\widetilde a)},$  where $\widetilde a$ is a lift of~$a$, and  the following commutative diagram \[\begin{tikzcd}
				{(\WW_G^{\leq m})^*} & {(\WW_H^{\leq m})^*} \\
				{\CCC_G(\leq m)} & {\CCC_H(\leq m)}
				\arrow[from=1-1, to=1-2]
				\arrow[from=1-1, to=2-1]
				\arrow[from=1-2, to=2-2]
				\arrow[from=2-1, to=2-2]
			\end{tikzcd}\]
		\end{proof}
		\begin{lem} Let $m\geq n\geq 1$. Given a homomorphism $\phi:G\to H$, for any $$a\in (\CCC_G(\leq m)\times_{\CCC_{\phi}(\leq m), \CCC_H(\leq m)}\CCC_G(\leq n))(\overline{\FF_q}) $$  one has that $[\eta_{G,a}]\subset (\CCC_G(\leq m)\times_{\CCC_{\phi}(\leq m), \CCC_H(\leq m)}\CCC_G(\leq n))(\overline {\FF_q})$.
		\end{lem}
		\begin{proof}
			For any such $a$, one has that $\CCC_\phi^{\leq m}([\eta_{G,a}])=[\eta_{R, \CCC_{\chi}({\leq m})(a)}]$. But as $\CCC_{\chi}({\leq m})(a)\in \CCC_R(\leq n)$ one must have $[\eta_{R, \CCC_{\chi}(\leq m)(a)}]\subset \CCC_R({\leq n})(\overline{\FF_q})$  and hence $$[\eta_{G,a}]\subset (\CCC_G(\leq m)\times_{\CCC_{\phi}(\leq m), \CCC_H(\leq m)}\CCC_G(\leq n))(\overline {\FF_q}).$$
		\end{proof}
		\subsubsection{}Let $k/\FF_q$ be an extension with $k$-perfect.
		\begin{prop}\label{condmostn}
			The equivalence $$\psi_G:BG\times \varinjlim_n \CCC_G(\leq n)^{\infty}\xrightarrow{\sim} \Delta_G $$identifies the group $$ (BG\times \CCC_{G}(\leq n)^{\infty})\langle k\rangle=BG\langle k\rangle\times \CCC_G(\leq n)(k)$$ with the subgroup $$\mathcal C_G(\leq n)\langle k\rangle\subset\Delta_G\langle k\rangle=BG\langle k\llparenthesis t \rrparenthesis\rangle $$ given by the $G$-torsors with the conductor no more than $n.$ Here, for a $G$-torsor we define its  conductor to be the conductor of any of the fields corresponding to some of its connected components. 
		\end{prop}
		\begin{proof}
			Let us treat the case when $G=\ZZ/p^e\ZZ$ is cyclic, because the general case is then deduced immediately by taking products and observing that for $G_1\times\cdots G_m$-torsor its conductor is the maximum of the conductors of the corresponding $G_i$-torsors for $i=1\doots m$. The map $\psi\langle k\rangle$ is a homomorphism, hence it suffices to show that the image of $(BG\times \CCC_G(\leq n)^{\infty})\langle k\rangle$ is the set of~$G$-torsors of the conductor no more than~$n$. Moreover, the map is bijective, hence we need to verify that the image of the former set is contained in the latter and that the preimage of the latter is the former. We have a commutative diagram:
			\[\begin{tikzcd}
				{\{0\}\times \CCC_G(\leq n)(k)} & {\Delta_G(k)} \\
				{\{0\}\times \CCC_G(\leq n)(\overline k)} & {\Delta_G(\overline k)},
				\arrow[from=1-1, to=1-2]
				\arrow[from=1-1, to=2-1]
				\arrow[from=1-2, to=2-2]
				\arrow[from=2-1, to=2-2]
			\end{tikzcd}\]
			and as $BG\langle k\rangle \times \CCC(\leq n)(k)$ and $\mathcal C_G(\leq n)\langle k\rangle$ are the preimages of $\{0\}\times \CCC_G(\leq n)(\overline k)$ and $\mathcal C_G(\leq n)\langle\overline k\rangle$ for the vertical maps, it suffices to prove the claim for the case $k$ is algebraically closed. Let us verify the first claim. Clearly, the image of $BG\langle k\rangle\times\{0\}$ is contained in $\CCC_G(\leq n)(k)$ and we need to show that the image of $0\times \CCC_G(\leq n)(k)$ is contained in $\mathcal C_G(\leq n)\langle k\rangle $.  An element in $\CCC_G(\leq n)(k)$ corresponds to an element $$(a_{ij})_{i,j}\in(\WW_G^{\leq n})^*(k)=\prod_{j=0}^{e-1}k^{\{i\in\NN|\hspace{0,1cm}i\leq (n-1)/p^{e-1-j}\}}.$$ The image of $(a_{ij})_{i,j}$ under~$\psi\langle k\rangle$ is given by the isomorphism class of $$\bigg(\sum_{i\in\NN'}a_{ij}t^{-ip^r}\bigg)_{j},$$where~$r$ can be any positive integer. By taking $p^r$-th roots of $a^i_j$, which we can do in the perfect field~$k$, we obtain that the last element is isomorphic to an element of the form $$\bigg(\sum_{i'\in\NN'}(a_{i'j})'t^{-i'}\bigg)_{j}$$with the condition that $(a_{i'j})'=0$ whenever $i>(n-1)/p^{e-1-j}$.   
			If the image of the last element is a field in $BG\langle k\llparenthesis t \rrparenthesis\rangle=H^1(k\llparenthesis t \rrparenthesis,G),$ we are done by \cite[Theorem 3.9]{tanno-yasuda}. If its image is not a field, then its image is the image under the canonical map $\Delta_H\langle k\rangle\to \Delta_G\langle k\rangle$ of a field for an injection $H\hookrightarrow G$. 
			By the commutativity of the diagram (implied by Theorem~\ref{comdeltag})
			\[\begin{tikzcd}
				{\{0\}\times \CCC_H(\leq n)(k)} & {\Delta_H\langle k\rangle} \\
				{\{0\}\times \CCC_G(\leq n)(k)} & {\Delta_G\langle k\rangle}
				\arrow[from=1-1, to=1-2]
				\arrow[from=1-1, to=2-1]
				\arrow[from=1-2, to=2-2]
				\arrow[from=2-1, to=2-2]
			\end{tikzcd}\]
			and the defining fact that the right vertical map preserves conductors, we are done by a simple induction. Let us now show that any element in $\mathcal C_G(\leq n)\langle k\rangle$ is the image of an element in $BG\langle k\rangle\times \CCC_G(\leq n)(k)$. For an element $a$, which is not a field, we have that it is the image of a field for the canonical map $\Delta_H\langle k\llparenthesis t \rrparenthesis\rangle\to \Delta_G\langle k\rangle$ hence the claim would follow by induction. Let us take~$a$ to be a field. We can write $a=a_0+a'$ with $a_0\in BG\langle k\llparenthesis t \rrparenthesis\rangle$ and $a'\in \mathcal C_G(\leq n)\langle k\rangle$. Then by \cite[Lemma 2.4 and Theorem 3.9]{tanno-yasuda}, the element $a'\in H^1(k\llparenthesis t \rrparenthesis,G)$ is the image of an element of form $$\widehat a:=\bigg(\sum _{i\in\NN'}(a_{ij})'t^{-i}\bigg)_j,$$ with $(a_{ij})'=0$ whenever $i\geq (n-1)/p^{e-1-j}$, for the canonical map $W_G(k\llparenthesis t \rrparenthesis)\to H^1(k\llparenthesis t \rrparenthesis,G.)$ It suffices to see that the element~$\widehat a$ is in the image of $\psi\langle k\rangle.$  It is isomorphic to the element $$\bigg(\sum_{i\in\NN'}((a_{ij})')^p t^{-pi}\bigg) $$
			which is in the image under~$\psi$ of $(((a_{ij})')^p)_{i,j}\in\CCC_G(\leq n)(k)$. The proof is completed.
		\end{proof}
		We will denote by~$\mathcal C_G(\leq n)$ the substack of~$\Delta_G$ corresponding to the image of $BG\times \CCC_G(\leq n)^{\infty}$ under~$\psi_G$, and notation will be consistent with the one in the above proposition.
		\section{Discriminant loci}\label{Discriminantloci}In this section, we study the loci where the discriminant is constant. Let $G$ be a finite abelian $p$-group.
		\subsection{Long flags}\subsubsection{}Let~$G^*$ be a group of characters, that is of homomorphisms $\chi:G\to S^1$. Each of this homomorphism factorizes as $\chi :G\to\ZZ/p^{e_{\chi}}\ZZ,$ for some $e_{\chi}\geq 1$, composed with $$\ZZ/p^{e_\chi}\ZZ\to S^1,\hspace{1cm} x+p^{e_\chi}\ZZ\mapsto \exp(2i\pi x/p^{e}).$$
		\begin{mydef}
			A long flag is a sequence $S=(S_n)_{n\in\NN\cup\{0\}} $ with $S_n\subset G^*$ subgroup satisfying that $S_{m}\supset S_{n}$ whenever $m\geq n$ and $S_n=G^*$ for~$n$ large enough. A jump of a long flag is an index $n\in\NN\cup\{0\}$ such that $S_n\supsetneq S_{n-1}$ with convention that $S_{-1}=\emptyset$. The finite set of jumps of~$S$ will be denote by $J(S)$.
		\end{mydef}
		One has that $\# J(S)\leq \# G.$ 
		\begin{mydef}
			Given a long flag~$S$, denote by $\disc(S)$ the number
			\begin{equation*}\disc(S):=\sum_{n\in\NN} n\cdot (\# S_n-\# S_{n-1})=\sum_{n\in J(S)}n\cdot (\# S_n-\# S_{n-1}).
			\end{equation*}
		\end{mydef}
		\begin{lem}\label{estdisc}
			For any long flag~$S$, one has that $$\disc(S)\leq \max(J(S))\cdot \# G(1-1/p).$$
		\end{lem}
		\begin{proof}
			One has that 
			\begin{align*}
				\disc(S)&=\sum_{n\in J(S)}n\cdot(\# S_n-\# S_{n-1})\\&\leq \max (J(S))\cdot (\# G^*-\# S_{n-1})\\&\leq \max (J(S))\cdot \# G^*\cdot (1-1/p).
			\end{align*}
		\end{proof}
		\begin{lem}
			Let $f\in\Delta_G$ and $m\geq 1$. One has that $$\{\chi\in G^*|\hspace{0,1cm}\Delta_{\chi}(f)\in\mathcal C_{\ZZ/p^{e_\chi}\ZZ}(\leq m)\} $$
			is a subgroup of $G^*$.
		\end{lem}
		\begin{proof}
			First, let us observe that for any integers $b\geq a$, one has for some $g\in\Delta_{\ZZ/p^a\ZZ}$ that $g\in\mathcal C_{\ZZ/p^{a}\ZZ}$ if and only if $\Delta_{i_{a,b}}(g)\in \mathcal C_{\ZZ/p^b\ZZ}(\leq m)$. Thus the set in the question coincides with $$\{\chi\in G^*|\hspace{0,1cm}\Delta_\chi(f)\in\mathcal C_{\Zpez}(\leq m)\},$$where $e$ is the exponent of~$G$. But $\mathcal C_{\Zpez}(\leq m)$ is a group stack, hence the claim follows.
		\end{proof}
		\begin{mydef}
			Given an element $f\in \Delta_G$, the sequence $$S_n(f):=\{\chi\in G^*|\hspace{0,1cm} \Delta_{\chi}(f)\in \mathcal C_{\ZZ/p^{e_{\chi}}\ZZ} (\leq n)\}$$ defines a long flag that we denote by $S(f)$. 
		\end{mydef}
		\begin{lem}
			Let~$\FF_{q'}/\FF_q$ be a finite extension and let~$f\in\Delta_G({\FF_{q'}})$. One has that $c_{\disc}(f)=\disc(S(f)).$ 
		\end{lem}
		\begin{proof}
			Applying the conductor discriminant formula: 
			$$c_{\disc}(f)=\sum_{\chi\in G^*}c_{\cond}(\Delta_{\chi}(f)).$$
			But for $n\in\NN$, the value of $c_{\cond}(\Delta_{\chi}(f))$ is~$n$ precisely for~$\chi\in S_n-S_{n-1}$. The statement follows. 
		\end{proof}
				\begin{lem}\label{numbofflag}
					The number of long flags~$S$ with $\disc(S)=n$ is $O(n^{\#G})$. 
				\end{lem}
				\begin{proof}
					A long flag~$S$ with $\disc(S)=n$ is completely determined by $(j, S_j)$ with $j\in J(S)$. The set $J(S)$ is a set of no more than $\# G$ positive integers of size no more than~$n$. The number of such sets is ${n \choose J(S)}$ and if~$n$ is sufficiently large the number is no more than ${n\choose \#G}$. The set $\{S_j|\hspace{0,1cm}_{j\in J(S)}\}$ is a set of~$J(S)$ subgroups of~$G^*$ and let~$A$ be the number of such sets. We deduce that the number of long flags~$S$ is no more than $A\cdot {n\choose {\# G}}=O(n^{\# G}).$ The statement is proven. %
				\end{proof}
		For $m\geq 0$ we define $$\mathcal C_G(m):=\begin{cases}
		\mathcal C_G(\leq m)-\mathcal C_G(\leq m-1)&\text{ if $m\geq 1$}\\
		\mathcal C_G(0)=0&\text{ if $m=0$.}
	\end{cases} $$
				\begin{lem}\label{condstays}
					Given $m\geq 1$ and $f\in\mathcal C_G(m)$, there exists a character $\chi: G\to\ZZ/p^{e_\chi}\ZZ$ such that $$\Delta_{\chi}(f)\in\mathcal C_{\ZZ/p^{e}\ZZ}(m).$$
				\end{lem}
				\begin{proof}
					If $G=\prod_{u=1}^d\ZZ/p^{e_u}\ZZ$ is a finite product of cyclic $p$-groups, we verify that for~$\chi$  one can take one of~$u$ projections $\chi:G\to\ZZ/p^{e_u}\ZZ$ composed with the canonical inclusion $\ZZ/p^{e_u}\ZZ$. 
					The type of a projection $\chi:G\to\ZZ/p^{e_u}\ZZ$ is $$((\overline{1},\Id_{e_u,e_u})_{i=u},(0, \rho_{e_i,e_u})_{i\neq u})_{i=1}^d,$$where $\overline 1$ stands for the class of $1\in\ZZ$ in $\ZZ/p^{e_u}\ZZ.$ 
					One has that \begin{align*}\Delta_{\chi}|_{BG\times\CCC_G(\leq m)^{\infty}}\hskip-2cm&\\&=B\chi\times\bigg(+^d\circ (((1\times -)\circ \CCC_{\Id_{e_u,e_u}}(\leq m)^{\infty})_{u}\circ (0)_{i\neq u})_{i=1}^d\bigg)\\
						&=B\chi\times\bigg(+^d\circ ((\CCC_{\Id_{e_u,e_u}}(\leq m)^{\infty})_{i=u},(0)_{i\neq u})_{i=1}^d\bigg)\\
						&=B\chi\times\bigg(\pr_u :\CCC_G(\leq m)^{\infty}=\prod_{i=1}^d\CCC_{\ZZ/p^{e_u}\ZZ}(\leq m)^{\infty}\to \CCC_{\ZZ/p^{e_u}\ZZ}(\leq m)^{\infty}\bigg),
					\end{align*}
					where $\pr_u$ stands for the projection to the $u$-th coordinate. Now, if for all~$u$ one has $\Delta_{\chi}(f)\in\CCC_{\ZZ/p^{e_u}\ZZ}(\leq m-1)$, then $f\in \prod_{u=1}^d\CCC_{\ZZ/p^{e_u}\ZZ}(\leq m-1)= \CCC_{G}(\leq m-1), $ a contradiction.
				\end{proof}
	\begin{mydef}
	Given a long flag $T=(T_n)_{n\in\NN\cup\{0\}}$, denote by 
	\begin{align*}
		Z(T):&=\bigcap_{n\in J(T)}\bigcap_{\chi\in T_n-T_{n-1}}\Delta_G\times_{\Delta_{\chi},\Delta_{\ZZ/p^{e_\chi}\ZZ}}\mathcal C_{\ZZ/p^{e_\chi}\ZZ}(n).
	\end{align*}
\end{mydef}
If $T=(T_n)_{n\in\NN}$ is a flag, then for every $n\in J(T)$ and every $\chi\in T_n-T_{n-1}$, one has that $$\Delta_G\times_{\Delta_{\chi},\Delta_{\Zpez}}\mathcal C_{\ZZ/p^{e_\chi}\ZZ}(n)\subset \mathcal C_G(\leq \max(J(T)))$$ and we deduce that 
\begin{equation} \label{ztschem} Z(T)\subset \mathcal C_G(\leq  \max({ J(T)})).\end{equation}
	\begin{lem}\label{finsupdisc}
One has that $$\sup_{n\neq 0}\dfrac{1+\sup_{\substack{ T\text{ long flag }\\ \disc(T)=n}}\dim(Z(T))}{n} <\infty$$
	\end{lem}
\begin{proof}
Given a long flag~$T$, we have an estimate $$\dim(Z(T))\leq \sum_{u=1}^d\sum_{j=0}^{e_u-1}\bigg(\bigg\lfloor\frac{\max(J(T))-1}{p^{e_u-j-1}}\bigg\rfloor-\bigg\lfloor\frac{\max(J(T))-1}{p^{e_u-j}}\bigg\rfloor\bigg)\leq C \max(J(T))$$ for certain $C>0$. On the other hand, for any long flag~$T$ with $\disc(T)=n,$ by Lemma~\ref{estdisc}, one must have that  $$\max(J(T))\geq \frac{n}{\# G\cdot (1-1/p)}.$$ one has that $$\sup_{n\neq 0}\frac{1+\sup_{\substack{T\text{ long flag }\\\disc(T)=n}}\dim(Z(T))}{n}\leq \frac{1+\frac{Cn}{\# G\cdot (1-1/p)}}{n}<\infty$$
\end{proof}
			\subsection{Irreducible components}
			\subsubsection{}
				%
			\begin{lem}
				Let $\chi:G=\prod_{u=1}^d\ZZ/p^{e_u}\ZZ\to \ZZ/p^h\ZZ=R$ be a character. Let $m\geq n\geq 0$. For the $\FF_q$-group scheme $$\mathfrak G(\chi,m,n):=\CCC_G(\leq m)\times_{\CCC_{\chi}(\leq m),\CCC_R(\leq m)}\CCC_R(\leq n)$$ one has that $$BG\times \mathfrak G(\chi,m,n)^{\infty}\cong\mathcal C_G(\leq m)\times_{\Delta_{\chi}|_{\mathcal C_G(\leq m)},\mathcal C_R(\leq m)}\mathcal C_R(\leq n).$$
				Moreover, 
				it is a connected unipotent algebraic group and satisfies a stronger property: for any $a\in\mathfrak G(\chi,m,n)$ one has that $[\eta_{G,a}]\subset\mathfrak G(\chi,m,n)$. 
			\end{lem}
			\begin{proof}
				One has that \begin{multline*}\mathcal C_G(\leq m)\times_{\Delta_{\chi}|_{\mathcal C_G(\leq m),\mathcal C_R(\leq m)}}\mathcal C_R(\leq n)\\=BG\times\bigg(\CCC_G(\leq m)^{\infty}\times_{\CCC_{\chi}(\leq m)^{\infty},\CCC_R(\leq m)^{\infty}}\CCC_R(\leq n)^{\infty}\bigg).
				\end{multline*}
				We have that \begin{align*}
					\CCC_G(\leq m)^{\infty}\times_{\CCC_{\chi}(\leq m)^{\infty},\CCC_R(\leq m)^{\infty}}\CCC_R(\leq n)^{\infty}
					\hskip-1cm&\\&=\bigg(\CCC_G(\leq m)\times_{\CCC_{\chi}(\leq m),\CCC_R(\leq m)}\CCC_R(\leq n)\bigg)^{\infty}\\
					&=(\mathfrak G(\chi,m,n))^{\infty}
				\end{align*}
				and the first claim follows. Let $a\in\CCC_G(\leq m)(\overline {\FF_q})$. One has that $$\CCC_{\chi}(\leq m)([\eta_{G,a}])=[\eta_{R,\CCC_{\chi}([a])}].$$ If $\CCC_{\chi}(a)\in\CCC_R(\leq n)$ then $$[\eta_{R,\CCC_{\chi}(a)}]\subset\CCC_R(\leq n).$$ But in this case, we get $[\eta_{G,a}]\subset\mathfrak G(\chi,m,n). $ Hence any $a\in\mathfrak G(\chi,m,n)$ satisfies that $[\eta_{G,a}]\subset\mathfrak G(\chi,m,n)$. We deduce that $\mathfrak G(\chi,m,n)$ is connected and hence irreducible. 
						\end{proof}
						\begin{prop}\label{irrflag}
							For every long flag, one has that $Z(S)=BG\times Z'(S),$ with $Z'(S)$ an irreducible open subset of an affine space or empty.
						\end{prop}
						\begin{proof}
							Let $m=\max(J(T))$. Thanks to~\ref{ztschem}, we have that \begin{align*}
								Z(T)&=\bigcap_{n\in J(T)}\bigcap_{\chi\in T_n-T_{n-1}}\Delta_G\times _{\Delta_{\chi},\Delta_R}\mathcal C_R(n)\\
								&=\bigcap_{n\in J(T)}\bigcap_{\chi\in T_n-T_{n-1}}\mathcal C_G(\leq m)\times _{\Delta_{\chi}|_{\mathcal C_G(\leq m)},\mathcal C_R(\leq m)}\mathcal C_R(n)\\
								&=\bigg(\bigcap _{n\in J(T)}\bigcap_{\chi\in T_n-T_{n-1}}\mathcal C_G(\leq m)\times_{\Delta_{\chi}|_{\mathcal C_G(\leq m)},\mathcal C_R(\leq m)}\mathcal C_R(\leq n)\bigg)\\&\quad\quad\quad\quad\setminus \bigg(\bigcup _{n\in J(T)}\bigcup_{\chi\in T_n-T_{n-1}}\mathcal C_G(\leq m)\times_{\Delta_{\chi}|_{\mathcal C_G(\leq m)},\mathcal C_R({\leq m})}\mathcal C_R(\leq n-1)\bigg)
							\end{align*}
							with convention $\mathcal C_R(\leq -1)=\emptyset$.
							One has that \begin{multline*}\bigcap_{n\in J(T)}\bigcap_{\chi\in T_n-T_{n-1}}\mathcal C_G(\leq m)\times _{\Delta_{\chi}|_{\mathcal C_G(\leq m)},\mathcal C_R(\leq m)}\mathcal C_R(n)\\
								=\bigcap_{n\in J(T)}\bigcap_{\chi\in T_n-T_{n-1}} BG\times (\mathfrak G(\chi, m,n))^{\infty}
							\end{multline*} and 
							\begin{multline*}
								\bigcup _{n\in J(T)}\bigcup_{\chi\in T_n-T_{n-1}}\mathcal C_G(\leq m)\times_{\Delta_{\chi}|_{\mathcal C_G(\leq m)},\mathcal C_R({\leq m})}\mathcal C_R(\leq n-1)\\=\bigcup _{n\in J(T)}\bigcup_{\chi\in T_n-T_{n-1}}	BG\times \mathfrak G(\chi,m,n)^{\infty}
							\end{multline*}
							For topological purposes we can ignore $BG$ and the perfections. 	 It follows from Lemma~\ref{quotpath} that $$\mathfrak G(m,n):=\bigcap_{n\in J(T)}\bigcap_{\chi\in T_n-T_{n-1}}\mathfrak G(\chi,m,n)$$ is a connected unipotent group because for any $a$ contained in the intersection of spaces $\mathfrak G(\chi,m,n),$ 
							one has that $[\eta_{G,a}]$ is also contained there, which is a connected subspace containing both $0$ and~$a$. 		To conclude, we observe that the space $Z(T) $ is obtained from the irreducible space $\mathfrak G(m,n)$ by removing closed subvarieties, so it is either empty or irreducible. 
						\end{proof}
	\begin{mydef}
Let us define $$\mathfrak D(n):=\bigcup_{\substack{ T\text{ long flag}\\ \disc (T)=n}}Z(T).$$
	\end{mydef}
					\begin{thm}\label{discirred}
			The constructible set $\mathfrak D(n)$ a union of no more than $O(n^{\# G})$ irreducible components which are open subspaces of affine spaces. 
					\end{thm}
					\begin{proof}
						This follows by combining Lemma~\ref{numbofflag} and Proposition~\ref{irrflag}.
					\end{proof}
\section{Commutative case}\label{npbh} Let~$F$ be a global field of characteristic $p>0$. For each place~$v$ in the set of places~$M_F$, let $F_v$ be the completion of~$F$ at~$v$ and $\OO_v$ its ring of integers. Let~$q_v$ be the cardinality of the residue field at~$v$. For a prime power~$r$, we denote by $\FF_r$ the finite field of cardinality~$r$. For every~$v$, let us fix an isomorphism $F\xrightarrow{\sim}\FF_{q_v}\llparenthesis t \rrparenthesis$ with the field of Laurent series over the finite field~$\FF_{q_v}$. Let $G$ be a finite abelian $p$-group.  
\subsection{Heights}We define heights that we work with.\subsubsection{}We recall some notions from \cite{dardayasuda2025}.
\begin{mydef}[{\cite[Definition 4.1]{dardayasuda2025}}]\label{defrais}
A raising function $c:|\Delta_G|\to\RR_{\geq 0}$ is a constructible function satisfying the following properties:
\begin{enumerate}
	\item one has that $c(x)=0$ if and only if $x=0$.
	\item one has that~$c$ is invariant for the isomorphisms of $|\Delta_G|$ which are induced from an automorphism $\FF_{q'}\llparenthesis t\rrparenthesis\xrightarrow{\sim}\FF_{q'}\llparenthesis t \rrparenthesis$.
\end{enumerate}
\end{mydef}
Given a field extension $K/\FF_q$ and $x\in \Delta_G(K)$ or $x\in \Delta_G\langle K\rangle$, we may write $c(x)$ for~$c(\widetilde x)$, where $\widetilde x$ is the point of $|\Delta_G|$ corresponding to~$x$. We may also write $c(\Delta_G)$ for what is formally $c(|\Delta_G|)$.
\begin{rem}
\normalfont The definition is identical to the one in \cite{dardayasuda2025}, except that we ask that~$c$ takes values in $\RR_{\geq 0}$, while there it is allowed to take the values in $\RR\cup\{\infty\}$ and that we ask that~$c$ is constructible (fibers are quasi--compact), while in \cite{dardayasuda2025} it is only asked to be locally constructible. In \cite[Definition 4.1]{dardayasuda2025} the word ``pseudo-raising'' is used, but in our situation the two notions coincide.
\end{rem}
\begin{rem}
	\normalfont We may use some abuse of notation and write $c:\Delta_G\to\RR_{\geq 0}$ or $c:\Delta_G(K)\to\RR_{\geq 0}$ where $K/\FF_q$ is an extension.
\end{rem}
\begin{rem}
	\normalfont 
In 
 \cite[Lemma 8.6 and Example 4.3]{dardayasuda2025} it is verified that 
 the conductor exponent and the discriminant exponent are examples of raising functions with~$c$ assumed only to be locally constructible. We will verify below that they are constructible and hence raising functions in the sense of Definition~\ref{defrais}.
\end{rem}
\begin{mydef}[{\cite[Definition 4.2]{dardayasuda2025}}] Let $c:|\Delta_G|\to\RR_{\geq 0}$ be a raising function. For $v\in M_F$, denote by $i_v:BG\langle F\rangle\to BG\langle F_v\rangle$ the canonical map. We define a height function $H=H(c),$ by $$H:BG\langle F\rangle \to \RR_{>0},\hspace{1cm}x\mapsto \prod_v q_v^{c(i_v(x))}.$$
We may write $H_v$ for $H_v:BG\langle F_v\rangle\to\RR_{>0}$ given by $x\mapsto q_v^{c(x)}$. 
\end{mydef}
\subsubsection{} 
For a discrete set~$X$, we denote $\mu^{\#}_X$ the counting measure on~$X$. We may omit~$X$ from the index, if the set~$X$ is clear from the context. For a place~$v$, we define a Haar measure~$\mu_v$ on~$BG\langle F_v\rangle$ by setting it to be $(\# G)^{-1}\cdot \mu^{\#}_{BG\langle F_v\rangle}$. We note that \begin{multline*}\mu_v(BG\langle\OO_v\rangle)=\frac{\#\{BG\langle \OO_v\rangle\}}{\# G}=\frac{\# H^1(\FF_{q}, G)}{\# G}=\frac{\# \Hom(\Gal(\overline {\FF_q}/\FF_q),G)}{\# G}\\=\frac{\#\Hom (\widehat{\ZZ}, G)}{\# G}=1.
	\end{multline*}
\begin{mydef}
	We define $$BG\langle\AAF\rangle:=\sideset{}{'}\prod_{v} BG\langle F_v\rangle,$$ where the restricted product is with respect to the finite subgroups $BG\langle\Ov\rangle\subset BG\langle F_v\rangle$.
\end{mydef}
\begin{rem}
\normalfont Note that $BG\langle \AAF \rangle$ does {\it not} mean the isomorphism classes of morphisms $\Spec(\AAF)\to BG$.
\end{rem}
One has that $BG\langle\AAA_F\rangle$ is a locally compact abelian group. Let~$\mu$ be the Haar measure on~$BG\langle \AAF\rangle$ given by $\bigotimes _v\mu_v$. 
\begin{prop}\begin{enumerate}
	
\item {\normalfont 	\cite[Proposition 4.12]{Cesnaviciuss}} The diagonal map $i:BG\langle F\rangle\to BG\langle\AAF\rangle$ is of closed, discrete and cocompact image.
\item {\normalfont\cite[Lemma 4.4(b)]{Cesnaviciuss}} The kernel of~$i$ is the finite group $\Sh^1(F,G):=\ker (i).$
\item The group $BG\langle\AAF\rangle/\big((i(BG\langle F\rangle) \prod_vBG\langle\Ov\rangle)\big) $ is finite.
\end{enumerate}
\end{prop}
\begin{proof}
We prove 3. By the definition of topology on~$BG\langle\AAF\rangle$, one has that the compact subgroup $\prod_vBG\langle\Ov\rangle\subset BG\langle\AAF\rangle$ is also open. Now the claim follows from~(2).   
\end{proof} 
Recall that the group $BG\langle\AAF\rangle/i(BG\langle F\rangle)\prod_vBG\langle\OO_v\rangle$ is finite.  
\begin{mydef}\label{approachable}
	We let $m_G$ be the smallest non-negative integer~$m$ such that there exists a set of representatives $y_1=1, y_2\doots y_k\in BG\langle \AAF\rangle$, where $k\geq 1$, modulo $i(BG\langle F\rangle)\prod_vBG\langle \OO_v\rangle$ such that for all~$v\in M_F$ and all $i=1\doots k$, one has that $$(y_i)_v\in BG\langle\FF_{q_v}\rangle\times \CCC_G(\leq m)(\FF_{q_v})$$ for the chosen (and hence any) isomorphism $F_v\cong \FF_{q_v}\llparenthesis t \rrparenthesis$.   
\end{mydef}
\subsubsection{}We define a property that will be used in the proof. Let~$G$ be a finite abelian $p$-group. We have a canonical identification $$\Delta_G\langle\overline{\FF_q}\rangle=BG\langle\overline{\FF_q}\rangle\times\varinjlim_n\CCC_G(\leq n)^{\infty}(\overline{\FF_q})=\varinjlim_n\CCC_G(\leq n)(\overline\FF_q).$$ 
Thus $\Delta_G(\overline{\FF_q})$ is canonically identified with $\Delta_G\langle\overline{\FF_q}\rangle$ we may occasionally use the former notation.
\begin{mydef}
Let $m\geq 1$. We say that a raising function $c:|\Delta_G|\to \RR_{\geq 0}$ is $m$-approachable, if there exists $0<C_1<C_2$ such that for every $x\in\Delta_G(\overline{\FF_q})$, one has that \begin{align*}C_1&\leq \inf\{c(x+w)-c(x)|\hspace{0,1cm}x\in \Delta_G(\overline{\FF_q}), w\in \CCC_G(\leq m)(\overline{\FF_q})\}\\&\leq \sup\{c(x+w)-c(x)|\hspace{0,1cm}x\in \Delta_G(\overline{\FF_q}), w\in \CCC_G(\leq m)(\overline{\FF_q})\}\\&\leq C_2.
\end{align*} 
\end{mydef}
\begin{rem} \label{condapp}
\normalfont The conductor counting function, which is defined to be $m$ on $\CCC_G(\leq m)-\CCC_G(\leq m-1)$ is clearly $k$-approachable for any $k>0$.
\end{rem} 
\begin{lem}\label{discap}
The discriminant exponent~$c_{\disc}$ is $k$-approachable for any $k>0$.
\end{lem}
\begin{proof}
Let $c^H_{\cond}:\Delta_H\to\RR_{\geq 0}$ denotes the conductor exponent of a finite abelian $p$-group~$H$. Let $\chi:G\to\ZZ/p^{e_\chi}\ZZ$ be a character, where $e_{\chi}\geq 1$ is such that~$\chi$ is surjective. It follows from above 
that there exists a constant $C(\chi,m)$ so that $$|c^{\ZZ/p^{e_\chi}\ZZ}_{\cond}(y+m)-c^{\ZZ/p^{e_\chi}}_{\cond}(y)|\leq C(\chi, W)$$as soon as $m\in \Delta_{\chi}(\CCC_G(\leq m)(\overline{\FF_q}))$ and any $y\in \Delta_G(\overline{\FF_q}).$ Let $w\in \CCC_G(\leq m)(\overline{\FF_q})$ and $x\in \Delta_G(\overline{\FF_q})$. One has that:
\begin{align*}|c_{\disc}(x+w)-c_{\disc}(x)|&=\sum_{\chi\in G^*}|c^{\ZZ/p^{e_\chi}\ZZ}_{\cond}(\Delta_{\chi}(x+w))-c^{\ZZ/p^{e_\chi}\ZZ}_{\cond}(\Delta_{\chi}(x))|\\
	&=\sum_{\chi\in G^*}|c^{\ZZ/p^{e_\chi}\ZZ}_{\cond}(\Delta_\chi(x)+\Delta_{\chi}(w))-c^{\ZZ/p^{e_\chi}\ZZ}_{\cond}(\Delta_{\chi} (x))|\\
	&\leq \sum_{\chi\in G^*}C(\chi,W).
\end{align*}
The statement is proven.
\end{proof}
\begin{lem}\label{hrprime} Let~$c:\Delta_G\to\RR_{\geq 0}$ be a raising function which is $m_G$-approachable. Let~$y_1=1\doots y_r$ be a set of representatives as in Definition~\ref{approachable}. Let~$H:BG\langle\AAF\rangle\to\RR_{\geq 0}$ be the height induced by~$c$. For any $y_i$, one has that there exist $C_2>C_1>0$ with 
$$C_1<\frac{H(y_ix)}{H(x)}<C_2.$$
\end{lem}
\begin{proof}
	Clearly for each~$v$, one has that $$x\mapsto \frac{H_v((y_i)_vx_v)}{H_v(x_v)}$$ is bounded from below and above, by the definition of $m$-approachable. Multiplying over finitely many~$v$ for which $y_j\not\in BG(\OO_v)$, gives the wanted claim.
\end{proof}
\subsubsection{}Let $c:\Delta_G\to\RR_{\geq 0}$ be a raising function. Given $r\in c(\Delta_G)$, we write $c^{-1}(r)$ for the constructible subset of~$\Delta_G$ which is the preimage of~$c^{-1}(r)$. If $\FF_{q'}/\FF_q$ is a finite extension, we write $$c^{-1}(r)(\FF_{q'})$$ for the subcategory of $\Delta_G(\FF_{q'})$ given by $\Spec(\FF_{q'})\to\Delta_G$ which induce the elements of $c^{-1}(r)$ and we write $c^{-1}(r)\langle \FF_{q'}\rangle$ for the corresponding isomorphism classes. All the counts of elements of $\Delta_G\langle\FF_r\rangle$ will be understood with weights $\#  G^{-1}$.
\begin{mydef}\label{suitable}
A raising function $c:|\Delta_G|\to\RR_{\geq 0}$ will be called semisuitable, if the following conditions are satisfied:
\begin{enumerate} 
	\item Suppose that $\sup_{r\neq 0}\frac{1+\dim(c^{-1}(r))}{r}<\infty.$
	\item For every $\varepsilon>0$, there exists $C>0$, such that for all $r>0$ and all finite extension $\FF_{q'}/\FF_q$, one has that $$\# c^{-1}(r)\langle\FF_{q'}\rangle\leq C (q')^{\dim(c^{-1}(r))+\varepsilon r}.$$
	\item There exists $M>0$ such that $Mc(\Delta_G)
\subset\ZZ$.
\end{enumerate}
It is said to be suitable if:
\begin{enumerate}
	\setcounter{enumi}{3}
	\item  	if it is also  $m_G$-approachable.
\end{enumerate}
\end{mydef}
\begin{lem}\label{discsuitable}
	One has that the discriminant exponent $c_{\disc}:\Delta_G\to\RR_{\geq 0}$ is a suitable counting function.
\end{lem}
\begin{proof}
	We have verified in \cite[Example 3.4]{dardayasuda2025} that~$c_{\disc}$ is locally constructible and satisfying conditions (1) and (2) of Definition~\ref{defrais}. To see that~$c_{\disc}$ is constructible, and hence a raising function, it suffices to see that $c_{\disc}^{-1}(r)$ are quasi-compact, but this follows from Theorem~\ref{discirred}. We now verify that it is suitable. The first condition of Definition~\ref{suitable} is verified in Lemma~\ref{finsupdisc}. 	Let~$r\in c_{\disc}(\Delta_G)-\{0\}$. Let $\FF_{q'}/\FF_q$ be a finite extension of~$\FF_q$. It follows from Theorem~\ref{discirred}  that $c_{\disc}^{-1}(r)(\FF_{q'}) $ is a union of spaces which are $\FF_{q'}$-points of a union of no more than $O(n^{\# G})$ open subsets of affine spaces. We deduce that the second condition of Definition~\ref{suitable} is satisfied. 
	 The third condition is automatic, because $c(\Delta_G)\subset\NN_{0}$. The fourth condition of Definition~\ref{suitable}, has been verified in Lemma~\ref{discap}
\end{proof} 
\begin{mydef}\label{aandb}
For any semisuitable raising function, we let $$a(c):=\sup_{r\in c(\Delta_G)-\{0\}}\frac{1+\dim(c^{-1}(r))}{r}.$$
By the first property of Definition~\ref{suitable}, one has that $a(c)<\infty$. Let us denote by $$D(c):=\{r\in c(\Delta_G)-\{0\}|\hspace{0,1cm} 1+\dim(c^{-1}(r))=a(c)r\}.$$
We 
let \begin{multline*}b(c):=\sum_{r\in D(c)}\#\text{ irreducible components of }c^{-1}(r)\\\text{ of dimension equal to }\dim (c^{-1}(r).
\end{multline*}
Hence, formally if $D(c)$ is infinite, $b(c)=\infty$, while if $D(c)=\emptyset$, then $b(c)=0$. 
\end{mydef}
\begin{rem}
\normalfont Let us verify that the $a$- and $b$-invariants agree with the ones in \cite{dardayasuda2025}. From Remark 8.2 of that paper, it is clear that the $a$-invariants coincide. The following expression is given for $b$: $$b:=\#\{x\in |\Delta_G||\hspace{0,1cm} 1+\dim (\overline{\{x\}})=ac(x)\}.$$ But then clearly, one must have that $x\in c^{-1}(D(c))$ and for each $r\in D(c)$, precisely the generic points of components of maximal dimension are counted.
\end{rem}
\begin{mydef}\label{strongsuitdef}
A suitable raising function $c:\Delta_G\to\RR_{\geq 0}$ will be called strongly suitable if the set $D(c)$ is finite and non-empty and there exists $\varepsilon(c)>0$ such that for any $r'\in c(\Delta_G)-D(c)-\{0\}$, one has that $$ a(c)-\varepsilon(c)>(1+\dim(c^{-1}(r')))/r'.$$
\end{mydef}
\begin{lem}\label{condss}
The conductor exponent is a strongly suitable raising function.
\end{lem}
\begin{proof} Let us verify that the conductor exponent~$c_{\cond}$ is a strongly suitable raising function. Let $G=\prod_{u=1}^d\ZZ/p^{e_u}\ZZ$, with $e_u\geq 1$. We have verified the conditions (1) and (2) of Definition~\ref{defrais} in \cite[Lemma 8.6]{dardayasuda2025}. It follows from Proposition~\ref{condmostn} that~$c_{\cond}$ is constructible and hence a raising function. We have seen in Remark~\ref{condapp}, that it is $k$-approachable for any~$k>0$. For any $r\in\NN$, one has that $$\dim(c_{\cond}^{-1}(r))=\sum_{u=1}^d\sum_{j=0}^{e_u-1}\bigg(\bigg\lfloor\frac{r-1}{p^{e_u-1-j}}\bigg\rfloor -\bigg\lfloor\frac{r-1}{p^{e_u-j}}\bigg\rfloor\bigg).$$ Clearly, \begin{align*} \dim(c_{\cond}^{-1}(r+p^{e_d}))-\dim(c_{\cond}^{-1}(r))=\sum_{u=1}^d(p^{e_d-e_u}-1).
	\end{align*}
Note that $$\max_{1\leq r\leq p^{e_d}}\frac{1+\dim(c_{\cond}^{-1}(r))}{r}\geq \frac{1+\sum_{u=1}^d(p^{e_d-e_u}-1)}{p^{e_d}}>\frac{\sum_{u=1}^d(p^{e_d-e_u}-1)}{p^{e_d}}.$$
Now it is clear that for any $1\leq r\leq p^{e_d}$ and any $\ell\geq 1$ that \begin{align*}\frac{1+\dim(c_{\cond}^{-1}(r))}{r}&>\frac{1+\dim(c_{\cond}^{-1}(r))+\ell \sum_{u=1}^{d}(p^{e_d-e_u}-1)}{r+\ell p^{e_d}}\\&=\frac{1+\dim(c_{\cond}^{-1}(r+\ell p^{e_d}))}{r+\ell p^{e_d}}. 
	\end{align*}
Thus the supremum $$\sup_{r\in c_{\cond}(\Delta_G)-\{0\}}\frac{1+\dim(c^{-1}_{\cond}(r))}{r}$$is finite and attained. Moreover, one has that $D(c_{\cond})$ is finite, because for any~$r$, one has $$\lim_{\ell\to\infty}\frac{1+\dim(c^{-1}_{\cond}(r))+\ell \sum_{u=1}^d(p^{e_d-e_u}-1)}{r+\ell p^{e_d}}=\frac{\sum_{u=1}^d(p^{e_d-e_u}-1)}{p^{e_d}}. $$ Moreover, there exists $\varepsilon>0$, such that for any $r'\in c_{\cond}(\Delta_G)-D(c)-\{0\}$, we have that $a(c)-\varepsilon>(1+\dim(c_{\cond}^{-1}(r')))/r' .$
\end{proof}
\begin{rem}\label{newheights}
\normalfont We construct new strongly suitable counting functions, with arbitrarily big $b$-invariant. This explains why the ``pseudo--effective'' cones from \cite{dardayasuda2025} need to be infinite dimensional. 
Let $m>0$ be an integer. By the fact established in the Proof of Lemma~\ref{condss} that for $\ell\geq 1$ and $1\leq r\leq p^{e_d}$, one has that $$\frac{1+\dim(c^{-1}_{\cond}(r+\ell p^{e_d}))}{r+\ell p^{e_d}}=\frac{1+\dim(c^{-1}_{\cond}(r))+\ell\sum_{u=1}^d(p^{e_d-e_u}-1)}{r+\ell p^{e_d}},$$ we deduce that there exists a sequence $r_1<\cdots <r_m\in c(\Delta_G)-\{0\}$ such that \begin{equation}\label{pruj}
\frac{1+\dim(c_{\cond}^{-1}(r_1))}{r_1}>\cdots >\frac{1+\dim(c_{\cond}^{-1}(r_m))}{r_m}.\end{equation} Let us set $M:=\{r_1\doots r_m\}.$ For $t\in\QQ_{>0}$, we define a constructible function $c^t:|\Delta_G|\to\RR_{\geq 0}$ by $$c(x):=\begin{cases}
\dfrac{c_{\cond}(x)t}{1+\dim(c^{-1}_{\cond}(c_{\cond}(x)))},&\text{ if $x\not\in c^{-1}(M); $}\\
c_{\cond}(x)&\text{ otherwise.}\\
\end{cases}$$

By using inequalities~(\ref{pruj}) and the fact that $c_{\cond}(\Delta_G)\subset\ZZ$, that for $t\ll 1$, one has that $c^t(x)=c^t(y)$ if and only if $c_{\cond}(x)=c_{\cond}(y)$. 
For any $r\in M$, set $r':= rt/(1+\dim(c_{\cond}^{-1}(r)))$. We have that $$\dfrac{1+\dim(c^{-1}(r'))}{r'}=\dfrac{1+\dim(c^{-1}(r'))}{\frac{rt}{1+\dim(c_{\cond}^{-1}(r))}}=\frac{1+\dim(c_{\cond}^{-1}(r'))}{\frac{rt}{1+\dim(c_{\cond}^{-1}(r))}}=\frac{1}{rt}.$$
This goes to infinity when $t\to 0$. We deduce that for $t\ll 1$, one has that $D(c^t)=\{r_1'\doots r_m'\}$ and that its cardinality is~$m$. Moreover, each $c^{-1}(r)$ is irreducible. We deduce that the $b$-invariant is~$m$. It is immediately verified that~$c^t$ is strongly suitable. 
\end{rem}
\begin{rem}
\normalfont We do not know if the discriminant exponent is a strongly suitable raising function.
\end{rem}

\subsubsection{}
More examples of strongly suitable counting functions come from representations of the cyclic group $G=\ZZ/p\ZZ$ of order $p$, where $p$ is the characteristic of the ground field. For $ 1 \le i \le p$, let $V_i$ denote the unique $i$-dimensional indecomposable linear representation of $G$ over $\FF_q$ and define the representation $V=\bigoplus_{\lambda=1}^l V_{d_\lambda}$, $1\le d_\lambda \le p$ of dimension $d=\sum_{\lambda=1}^l d_\lambda$. Following \cite{Yasuda_2014}, we define its invariant $D_V$ by 
\[
D_V := \sum_{\lambda=1}^l \frac{(d_\lambda -1)d_\lambda}{2},
\]
and for a positive integer $j$ coprime to $p$,  define 
\[
\mathrm{sht}_V(j) :=\sum_{\lambda=1}^l \sum_{i=1}^{d_\lambda -1} \left \lfloor \frac{ij}{p} \right \rfloor. 
\]
(This is equal to $-\mathbf{w}$ with the notation from \cite{woodyasuda}.)
The $\mathbf{v}$-function on $|\Delta_G|$ associated to $V$ \cite[Definition 3.7]{woodyasuda} is then defined by 
\[
\mathbf{v}(x) =
\begin{cases}
	0 & (x=0) \\
	d-l + \mathrm{sht}_V(j) & (\text{$x$ has ramification jump $j$}).
\end{cases}
\]
This function is closely related to the geometry of the quotient variety associated to the representation $V$. 
Assuming that $D_V = p$, we will show that $\mathbf{v}$ is a strongly suitable raising function and compute $a(\mathbf{v})$ and $b(\mathbf{v})$ through Lemmas \ref{AtMostMinus1} to \ref{v-str-suit} below. 
Examples of representations $V$ with $D_V=p$ include:
\begin{itemize}
	\item $V_2^{\oplus p}$ ($\forall p$),
	\item $V_2^{\oplus 2}$ ($ p = 2 $),
	\item $V_3$ ($ p = 3 $),
	\item $V_3\oplus V_2^{\oplus 2}$ ($p=5$),
	\item $V_4 \oplus V_2$ ($p=7$).
\end{itemize}

For $j>0$ with $p\nmid j$, 
the locus $C_j \subset \Delta_G$ of points with  ramification jump $j$  has coarse moduli space which is the ind-perfection of $\mathbb{G}_m \times \mathbb{A}^{j-\lfloor j/p\rfloor -1 }$. In particular, it has dimension $j-\lfloor j/p \rfloor$. 
In what follows, we use the notation $\mathbf{v}(j) = d-l+\mathrm{sht}_V(j)$ for $j>0$ with $p\nmid j$ so that for $x$ with ramification jump $j > 0$, we have $\mathbf{v}(x)=\mathbf{v}(j)$. 

\begin{lem}\label{AtMostMinus1}
	If $D_V=p$, then for every $j>0$ with $p\nmid j$, we have
	\[
	\dim C_j -\mathbf{v}(j) \le -1. 
	\]
	Morever, we have 
	\[
	\dim C_{p-1} -\mathbf{v}(p-1) = -1. 
	\]
\end{lem}
\begin{proof}
	We have 
	\[
	\dim C_{j} -\mathbf{v}(j) = 
	j-\left\lfloor \frac{j}{p}\right\rfloor
	+l-d -\mathrm{sht}_V(j).
	\]
	From \cite[Proof of Proposition 6.9]{Yasuda_2014}, $\mathrm{sht}_V(j+p)$=$\mathrm{sht}_V(j)+p$. 
	Thus, 
	\begin{align*}
		\dim C_{j+p} -\mathbf{v}(j+p)& = 
		j+p-\left(\left\lfloor \frac{j}{p}\right\rfloor +1 \right)
		+l-d -(\mathrm{sht}_V(j)+p)\\
		&=\dim C_{j} -\mathbf{v}(j)-1.
	\end{align*}
	Therefore, $\dim C_{j} -\mathbf{v}(j)$ attains the maximum only at some elements of $\{1,2,\dots,p-1\}$. 
	From \cite[Lemmas 3.2 and 3.3]{Yasuda_discrepancies},  for $j\in \{1,2,\dots,p-1\}$, we have
	\begin{align*}
		&\dim C_{j} -\mathbf{v}(j)\\
		&= l-d+j-\mathrm{sht}_V(j) \\
		&= \mathrm{sht}_V(p-j)+j-p  \\
		&\le (p-j-1)+j-p  \\
		& = -1.
	\end{align*}
	We have proved the inequality of the lemma.
	For $j=p-1$, we have
	\begin{align*}
		\dim C_{p-1} -\mathbf{v}(p-1)
		= \mathrm{sht}_V(1)-1 = -1. 
	\end{align*}
\end{proof}

\begin{lem}
	We have 
	\begin{align*}
		a(\mathbf{v})&= 1,\\
		b (\mathbf{v})&= \# \{ j \in \{1,2,\dots, p-1\} \mid l-d+j-\mathrm{sht}_V(j)= -1 \}.
	\end{align*}
\end{lem}
\begin{proof}
	For every $r>0$, we have the equivalence
	\[
	\frac{1+\dim \mathbf{v}^{-1}(r)}{r} \le 1 \Leftrightarrow \dim \mathbf{v}^{-1} (r) - r \le -1.
	\]
	Moreover, the equality on the left is strict if and only if so is the one on the right. 
	From  the last lemma, the right inequality holds for every $r=\mathbf{v}(j)$ and the equality holds for $r=\mathbf{v}(p-1)$. This shows that $a(\mathbf{v})=1$.
	Moreover, from the proof of the last lemma, the equality $\mathbf{v}^{-1} (r) - r = -1$ fails for $j>p$. This shows the assertion for $b(\mathbf{v})$.
\end{proof}

\begin{lem}\label{v-str-suit}
	Suppose $D_V=p$. Then, $\mathbf{v}$ is strongly suitable. 
\end{lem}

\begin{proof}
	(1) and (3) of Definition \ref{suitable} are obvious. (2) follows from the fact that 
	the coarse moduli space of $c^{-1}(r)$ is an open subset of an affine space.
	
	We can show that $\mathbf{v}$ is $m$-approachable for any $m$ by using the fact that
	\[
	k\llparenthesis t \rrparenthesis / \wp k\llparenthesis t \rrparenthesis =\bigoplus _{j>0;p\nmid j} kt^{-j} \oplus k/\wp k
	\] 
	and a Laurent polynomial in $\bigoplus _{j>0;p\nmid j} kt^{-j} \oplus k/\wp k$ of order $-j$ corresponds to a torsor of ramification jump $j$. 
	
	Lastly, we show that $\mathbf{v}$ is strongly suitable. In the current situation, for $r'=\mathbf{v}(j)$, the inequality in Definition \ref{strongsuitdef} is equivalent to
	\[
	\dim C_j -\mathbf{v}(j)= \dim \mathbf{v}^{-1}(r')-r' < -1 -\varepsilon r'= -1 -\varepsilon \mathbf{v}(j).
	\]
	To have this inequality, it is enough to put  $\varepsilon $ to be a positive real number less than
	\[
	\min \left \{\frac{1}{\mathbf{v}(j)+1} \mid 1\le j \le p-1 \right \}. 
	\]
	Indeed, for $1\le j \le p-1$ with $\dim C_j -\mathbf{v}(j)<-1$, we have
	\[
	\dim C_j-\mathbf{v}(j) \le -2 < 
	-1 - \varepsilon \mathbf{v}(j) .
	\]
	For $j>p$, we write $j=j'+np$, $1\le j' \le p-1$, $n >0$. Then, since  
	\[
	\frac{1}{\mathbf{v}(j')+1} \le  \frac{np}{\mathbf{v}(j')+np} = \frac{np}{\mathbf{v}(j)},
	\]
	we get
	\begin{align*}
		\dim C_j-\mathbf{v}(j)
		&= \dim C_{j'}-\mathbf{v}(j')-np \\
		&\le -1 -np  \\
		&=  -1 - \frac{np}{\mathbf{v}(j)}\mathbf{v}(j) \\
		& <  -1 - \varepsilon \mathbf{v}(j).
	\end{align*}
\end{proof}

\subsubsection{}We now let $c:\Delta_G\to\RR_{\geq 0}$ be a strongly suitable raising function. For every $r\in c(\Delta_G)-\{0\}$, let $W_r$ be the set of irreducible components of $c^{-1}(r)$ of dimension $\dim(c^{-1}(r))$. For $r\in D(c)$, let us set $$b_{r,q'}:=\# (W_r^{\Gal(\overline{\FF_q}/\FF_{q'})})$$ and $$R_{r, q'}:=b_{r, q'}\cdot (q')^{\dim(c^{-1}(r))}- \#(c^{-1}(r)(\FF_{q'})).$$
 It follows from Lang--Weil estimates~\cite[Theorem 1]{LangWeil} and the fact that $$c^{-1}(r)\langle\FF_{q'}\rangle=BG\langle\FF_{q'}\rangle\times X(\FF_{q'}),$$ with~$X$ finite union of locally closed subsets of a variety, that for each $r\in D(c)$, there exists a constant~$Q_r$ such that $$|R_{r,q'}|\leq Q_r\cdot (q')^{\dim(c^{-1}(r))-\frac 12}.$$
\begin{lem}\label{zsrn} Suppose that $c$ is strongly suitable.
There exists $\varepsilon'>0$ such that $$\sum_{r\in D(c))}\# (c^{-1}(r)\langle\FF_{q'}\rangle)\cdot (q')^{-sr}
	= O((q')^{-7/6})+\sum_{r\in D(c)}b_{r, q'}\cdot (q')^{\dim(c^{-1}(r))-sr}$$
	whenever $\Re(s)>a-\varepsilon'$.
\end{lem}
\begin{proof}
	If $\varepsilon'>0$ is sufficiently small, then $\dim(c^{-1}(r))-\Re(s)r<-\frac{2}{3}$ for $\Re(s)>a-\varepsilon'$ and every $r\in D(c)$. In this domain we have that 
	\begin{align*}
		\bigg|\sum_{r\in D(c)}\# c^{-1}(r)\langle\FF_{q'}\rangle(q')^{-sr}-\sum_{r\in D(c)}b_{r,q'}\cdot (q')^{\dim(c^{-1}(r))-sr}\bigg|\hskip-6cm&\\&\leq\sum_{r\in D(c)}(q')^{-\Re(s)r}|\# c^{-1}(r)(\FF_{q'})-b_{r,q'}\cdot (q')^{\dim(c^{-1}(r))}| \\
		&\leq \sum_{r\in D(c)}(q')^{-\Re(s)r}\cdot Q_r\cdot (q')^{\dim(c^{-1}(r))-\frac 12}\\
		&\leq \sum_{r\in D(c)}(q')^{\dim(c^{-1}(r))-\frac 12-\Re(s)r}\\
		&=O((q')^{-7/6}).
	\end{align*}
\end{proof}
\begin{prop}\label{localheightzeta}  Suppose that~$c$ is strongly suitable.  One has that  $$s\mapsto\int_{\Delta_G\langle\FF_{q'}\rangle}H_{q'}^{-s}\mu_{q'}$$ is a holomorphic function in the domain $\Re(s)>a-\min(\varepsilon(c)-\varepsilon_0,\varepsilon')$ for $\varepsilon'>0$ as in Lemma~\ref{zsrn} and one has $$\int_{\Delta_G\langle\FF_{q'}\rangle}H_{q'}^{-s}\mu_{q'}=1+\sum_{r\in D(c)}b_{r,q'}\cdot (q')^{\dim(\cc^{-1}(r))-sr} +O((q')^{-1-\delta})$$for certain $\delta>0$ in this domain. 
\end{prop}
\begin{proof}
Whenever all the quantities converge, by using Lemma~\ref{zsrn}, we obtain that in the domain $\Re(s)>a-\varepsilon'$ one has that:
	\begin{align*}\int_{\Delta_G\langle\FF_{q'}\rangle}H_{q'}^{-s}\mu_{q'}\hskip-3cm&\\&= 1+\sum_{r\in c(\Delta_G)-\{0\}}\# c^{-1}(r)\langle\FF_{q'}\rangle\cdot (q')^{-sr}\\
		&= 1+\sum_{r\in D(c)}\# c^{-1}(r)\langle\FF_{q'}\rangle\cdot (q')^{\dim(c^{-1}(r))-sr}+\hskip-0,4cm\sum_{r\not\in D(c)\cup\{0\}}\# (c^{-1}(r)\langle\FF_{q'}\rangle)(q')^{-sr}\\
		&=1+O((q')^{-7/6})+\sum_{r\in D(c)}b_{r,q'}(q')^{\dim(c^{-1}(r))-sr}+\hskip-0,4cm\sum_{r\not\in D(c)\cup\{0\}}\hskip-0,4cm\# (c^{-1}(r)\langle\FF_{q'}\rangle)(q')^{-sr}.
	\end{align*} By using that for $r\not\in D(c)\cup\{0\}$, one has that $1+\dim(c^{-1}(r))  <(a-\varepsilon(c))r,$ we obtain, whenever all quantities converge, that:
\begin{align*}
\bigg|\sum_{r\not\in D(c)\cup\{0\}}\# c^{-1}(r)\langle\FF_{q'}\rangle\cdot(q')^{-sr}\bigg|&\leq\sum_{r\not\in D(c)\cup\{0\}}\# c^{-1}(r)\langle\FF_{q'}\rangle\cdot (q')^{-\Re(s)r}\\
&\leq \sum_{r\not\in D(c)\cup\{0\}}C\cdot (q')^{\dim(c^{-1}(r))+\varepsilon_0 r-\Re(s)r}\\ 
&=\sum_{r\not\in D(c)\cup\{0\}}C\cdot (q')^{-1+(\varepsilon_0+a-\varepsilon(c)-\Re(s))r}.
\end{align*}
  By the assumption that there exists~$\varepsilon_1>0$ such that whenever $r\neq r'\in c(\Delta_G),$  then $|r-r'|>\varepsilon_1$, we deduce that in the domain $\Re(s)>a-(\varepsilon( c)-\varepsilon_0)$ the sum $$\sum_{r\not\in D( c)\cup\{0\}}\#  c^{-1}(r)(\FF_{q'})\cdot(q')^{-sr}$$ converges absolutely and uniformly to a holomorphic function. Moreover, in the absolute value, it can be bounded by $O((q')^{-1-\delta})$ for certain $\delta>0$. The claim is proven. 
\end{proof}
	%
\subsubsection{} We continue assuming that $c:\Delta_G\to\RR_{\geq 0}$ is strongly suitable. In this paragraph, we analyse the product of local height zeta functions. If~$v$ is a place of~$F$, we let $q_v$ be the cardinality of the residue field at~$F_v$ and we write~$H_{v}$ for $H_{q_v}$,~$b_{r,v}$ for~$b_{r,q_v}$, etc. \begin{prop}\label{brv}
	Let~$r\in D(c)$. The product \begin{equation}\label{lemprodbrv}\prod_v\bigg(1-b_{r,v}q_v^{\dim(c^{-1}(r))-sr}\bigg),
	\end{equation}converges absolutely to a holomorphic function in the domain $\Re(s)>(1+\dim(c^{-1}(r)))/r$ which extends to a holomorphic function in the domain $\Re(s)\geq (1+\dim( c^{-1}(r)))/r$ with a zero at $s=(1+\dim( c^{-1}(r)))/r$ of multiplicity $$\sum_{\substack{\mathcal Y\subset c^{-1}(r)\text{ irreducible}\\ \dim(\mathcal Y)=\dim( c^{-1}(r))}}1.$$
	Moreover, there exists $\varepsilon>0$ such that in the domain $\Re(s)>(1+\dim(c^{-1}(r)))r^{-1}-\varepsilon$, the function does not have poles outside of the set $\{(1+\dim(c^{-1}(r)))r^{-1}+2ik\pi|\hspace{0,1cm}k\in\ZZ\}$.
\end{prop}
\begin{proof}
	For almost all places~$v$, we will compare the $v$-adic factor in~(\ref{lemprodbrv}) with the $v$-adic factor of a translated~$L$-function that we will define. Recall that we defined a $\Gal(\overline{\FF_q}/\FF_q)$-set~$W_r$ as the set of irreducible geometric components of~$c^{-1}(r)$ of dimension $\dim(c^{-1}(r))$. 
	 There exists a finite extension $\FF_{q'}/\FF_q$ such that each element of~$W_r$ is defined over~$\FF_{q'}$, so~$W_r$ is $\Gal(\FF_{q'}/\FF_q)$-set.  Set $F':=\FF_{q'}F$. Then, by \cite[Proposition 8.3]{geometrichermite}, one has that~$\FF_{q'}$ is the field of constants of~$F'$.  
Now	\cite[Proposition 9.1]{geometrichermite} provides a canonical surjective homomorphism 
	$$\psi:\Gal(F'/F)\to\Gal(\FF_{q'}/\FF_q),\hspace{1cm}\psi(\sigma)=\sigma|_{\FF_{q'}}$$ and for a place~$v$ and its extension~$v'$ to~$F'$ a canonical surjective homomorphism $$\psi_v:\Gal(F'_{v'}/F_v)\to\Gal(\FF_{q_{v'}}/\FF_{q_v}),\hspace{1cm}\psi_v(\sigma)=\sigma|_{\FF_{q_{v'}}}$$ where we note that the fields of constants of~$F'_{v'}$ and~$F_v$ are~$\FF_{q_{v'}}$ and~$\FF_{q_v}$, respectively. Moreover, there is a commutativity relation $$\psi\circ (\Gal(F'_{v'}/F_v)\to\Gal(F'/F))= (\Gal(\FF_{q_{v'}}/\FF_{q_v})\to\Gal(\FF_{q'}/\FF_q))\circ\psi_v$$where the maps on Galois groups are the restrictions. The map~$\psi$ endows the set~$W_r$ with a structure of $\Gal(F'/F)$-set. Now, define a $\Gal(F'/F)$-representation to be the vector space~$\CC^{W_r}$ endowed with the action $$\gamma\cdot (z_{w})_w:=(z_{\gamma\cdot w})_w\hspace{1cm}\gamma\in\Gal(F'/F), (z_w)_w\in\CC^{W_r}.$$
	For any place~$v$ we let $\Frob_v\in\Gal(F'_{v'}/F_v)$ be a preimage under~$\psi_v$ of the Frobenius in~$\Gal(\FF_{q_v'}/\FF_{q_v})$. Then whenever~$W_r$ is unramified at~$v$, one has that: \begin{align*}W_r^{\Frob_v}=W_r^{\Gal(F'_{v'}/\Fv)}&=W_r^{\psi(\Gal(F'_{v'}/\Fv))}\\&=W_r^{((\Gal(\FF_{q_{v'}}/\FF_{q_v})\to\Gal(\FF_{q'}/\FF_q))(\psi_v(\Gal(\FF_{q_{v'}}/\FF_{q_v})))}\\&=W_r^{\Imm(\Gal(\FF_{q_{v'}}/\FF_{q_v})\to \Gal(\FF_{q'}/\FF_q))}\\
		&=W_r^{\Imm(\Gal(\FF_{q'}\FF_{q_v}/\FF_{q_v})\to\Gal(\FF_{q'}/\FF_q))}\\
		&=W_r^{\Gal(\FF_{q'}/(\FF_{q'}\cap \FF_{q_v}))}.\end{align*}
	Note that the last set is precisely the set of components of~$W_r$ which are defined over~$\FF_{q_v}$. We deduce that $$b_{r,v}=\# W_r^{\Frob_v}=\tr(\Frob_v),$$
	where in the last equality~$\Frob_v$ is seen as a linear map on~$\CC^{W_r}$ via the above representation. Now, it follows that in the domain $\Re(s)>\dim( c^{-1}(r))/r+\varepsilon_1$, where $\varepsilon_1$ is sufficiently small, for all~$v$ and all~$\delta>0$, one has that $$(1-b_{r,v}q_v^{\dim( c^{-1}(r))-sr})(1+\tr(\Frob_v)q_v^{\dim(c^{-1}(r))-sr}+O(q_v^{-1-\delta}))=1+O(q_v^{-1-\delta'})$$for some~$\delta'>0$. The second factor is the~$v$-adic part of the translated $L$-function $L(sr-\dim( c^{-1}(r)))$. Taking the product over all~$v$, and using the fact that for all~$v$ the functions $ 1-b_{r,v}q_v^{\dim( c^{-1}(r))-sr}$ are non-vanishing and holomorphic in the domain $\dim( c^{-1}(r))/r+\varepsilon_1$, we obtain that $$\prod_v(1-b_{r,v}q_v^{\dim( c^{-1}(r))-sr})$$ converges to a holomorphic function in the domain where the product of factors defining the translated $L$-function converges to a holomorphic function, that is, in $\Re(s)>(1+\dim(c^{-1}(r)))/r=a$. Moreover, one has that the product extends holomorphically to the domain where the translated $L$-function extends meromorphically and without zeros, that is, in $\Re(s)\geq (1+\dim(c^{-1}(r)))/r$. The order of the zero at $(1+\dim(c^{-1}(r)))/r$ coincides with the order of the pole of the translated $L$-function, which is precisely $$\# (W_r/\Gal(F'/F))= \# (W_r/\Gal(\FF_{q'}/\FF_q))= \sum_{\substack{\mathcal Y\subset c^{-1}(r)\text{ irreducible}\\ \dim(\mathcal Y)=\dim(c^{-1}(r))}}1.$$
	The possible poles in the domain $\Re(s)\geq (1+\dim(c^{-1}(r)))r^{-1}$ are in the set $$\{t|\hspace{0,1cm}rt-\dim(c^{-1}(r)) \in \{1+2ki\pi|\hspace{0,1cm}k\in\ZZ\}\},$$
	which is the set $$\{(1+\dim(c^{-1}(r)))r^{-1}+2ik\pi|\hspace{0,1cm}k\in\ZZ\}.$$
	The statement is proven.
\end{proof}
\begin{prop}\label{labav}
	The product $$\int_{BG\langle \AAF\rangle}H^{-s}\mu=\prod_{v\in M_F}\int_{BG\langle F_v\rangle}H_v^{-s}\mu_v$$converges absolutely to a holomorphic function in the domain $\Re(s)>a=a(c)$ and extends to a meromorphic function in the domain $\Re(s)\geq a$ with a pole at $s=a$ which is of order $b=b(c).$ Moreover, for some $\varepsilon>0$ the possible poles of the function in the domain $\Re(s)\geq a(c)-\varepsilon$ are in the set $\{a(c)+2ik\pi|\hspace{0,1cm}k\in\ZZ\}.$
\end{prop}
\begin{proof} 
	 Set~$\varepsilon_1:=\min(\varepsilon(c)-\varepsilon_0, \varepsilon')$, with the notation as in Proposition~\ref{localheightzeta}. Then this proposition gives that whenever $\Re(s)>a-\varepsilon_1$, one has that:
	\begin{equation}
		\label{ketr}
		\prod_v\int_{BG\langle F_v\rangle}H_v^{-s}\mu_v=\prod_v\bigg(1+\sum_{r\in D( c)}b_{r,v}q_v^{\dim(c^{-1}(r))-sr}+O(q_v^{-7/6})\bigg).
	\end{equation}
		We multiply this product by \begin{equation}\label{appprod}\prod_v\prod_{r\in  D(c)} \bigg(1-b_{r,v}q_v^{\dim(c^{-1}(r))-sr}\bigg),
		\end{equation}which, by Proposition~\ref{brv} converges absolutely to a holomorphic function in the domain $\Re(s)>a$ and which extends to a holomorphic function in the domain $\Re(s)\geq a$ with a zero at $s=a$ of multiplicity $\sum_{r\in D(c)}(\# W_r/\Gal(\overline {\FF_q}/\FF_q))=b$. Also, the set of its possible poles  in the domain $\Re(s)>a(c)-\varepsilon''$ for some $\varepsilon''>0$ is contained in the set $\{a(c)+2ik\pi|\hspace{0,1cm}k\in\ZZ\}$. For sufficiently small~$\varepsilon_2>0$ in the domain $\Re(s)>a-\varepsilon_2$ each term in~(\ref{appprod})
		expands to \begin{multline*}1-\sum_{r\in D(c)} b_{r,v}q_v^{\dim(c^{-1}(r))-sr}+\sum_{(r,t)\in (D(c))^2}\hskip-0,5cm b_{r,v}b_{t,v}q_v^{\dim(c^{-1}(r))+\dim(c^{-1}(t))-s(r+t)}\cdots\\=1-\sum_{r\in D(c)} b_{r,v}q_v^{\dim(c^{-1}(r))-sr}+O(q_v^{-1-\delta_0}).\end{multline*} for certain $\delta_0>0$. The multiplication produces a product of terms over~$v$ and each term in absolute value can be bounded by  
		\begin{multline*}
			\bigg|1+(O(q_v^{-1-\delta'})+O(q_v^{-1-\delta_0}))\bigg(\sum_{r\in D(c)}b_{r,v}q_v^{\dim(c^{-1}(r))-sr}\bigg)\bigg|\\
			=1+O(q_v^{-1-\max(\delta',\delta_0)})
		\end{multline*}
		for $\Re(s)>a-\min (\varepsilon_1,\varepsilon_2)$. The wanted claim follows.
	\end{proof}
\subsubsection{} We now prove a Northcott property. 
\begin{lem}\label{bounddisc}
There are only finitely many $G$-torsors over~$F$, which are unramified everywhere.
\end{lem}
\begin{proof}
This follows from \cite[Theorem 1.1]{geometrichermite} and the fact that there are only finitely many constant extensions of~$F$ of bounded degree.
\end{proof}
\begin{prop}\label{northcott}
	Suppose that $c:\Delta_G\to\RR_{\geq 0}$ is semi-suitable. There exists $C,d>0$ such that $$\#\{x\in BG\langle F\rangle|\hspace{0,1cm}H(x)\leq B\}\leq CB^{d}.$$
\end{prop}
\begin{proof}
Let us first establish that for every $B>0$, one has that $$\#\{x\in BG\langle F\rangle|\hspace{0,1cm}H(x)\leq B\}\ll \mu(\{x\in BG\langle \AAF\rangle|\hspace{0,1cm}H(x)\leq B\}).$$
First, one has that $$\#\{x\in BG\langle F\rangle|\hspace{0,1cm}H(x)\leq B\}\asymp \#\{x\in i(BG\langle F\rangle)|\hspace{0,1cm}H(x)\leq B\},$$because, as~$G$ is commutative, all the fibers of~$i$ have the same finite cardinality $\#\Sh^1(G)$. Let $K:=\prod_v BG\langle \OO_v\rangle$. By Lemma~\ref{bounddisc}, the set of elements in $i(BG\langle F\rangle)$ lying in~$K$ is finite. The height~$H$ is~$K$-invariant. The Haar measure~$\mu$ on~$BG\langle \AAF\rangle$, restricted to~$i(BG\langle F\rangle)K$ is, up to the constant the product of the probability measure on~$K$ and the discrete measure on~$i(BG\langle F\rangle)$. We deduce that $$\#\{x\in i(BG\langle F\rangle)|\hspace{0,1cm}H(x)\leq B\}\asymp \mu(\{x\in i(BG\langle F\rangle)K|\hspace{0,1cm}H(x)\leq B\}).$$  
It is clear that $$\mu(\{x\in i(BG\langle F\rangle)K|\hspace{0,1cm}H(x)\leq B\})\leq \mu(\{x\in BG\langle \AAF\rangle|\hspace{0,1cm}H(x)\leq B\}),$$hence  one has $$ \#\{x\in BG\langle F\rangle|\hspace{0,1cm}H(x)\leq B\}\ll\mu(\{x\in BG\langle \AAF\rangle|\hspace{0,1cm}H(x)\leq B\})$$as we have claimed. To prove the claim of the statement, we will show that $$\mu(\{x\in BG\langle \AAF\rangle|\hspace{0,1cm}H(x)\leq B\})\ll B^d$$ for certain $d\geq 1$. We are clearly allowed to change the function~$c$ by $\lambda c$, where $\lambda>1$ is such that for any $r_1\neq r_2\in c(\Delta_G)$ with $|\lambda c(r_1)-\lambda c(r_2)|>1$.  
 Given an ideal $I\subset \OO_F$ and a place~$v$, we write $\alpha_v(I)$ for the exponent of  the prime corresponding to~$v$ in the prime factorization of~$I$. We define a subset $V(I)$ of $BG\langle \AAF\rangle$ by asking that $H_v(x_v)\leq q_v^{{\alpha_v(I)}}.$ 
 We verify that the sets $V(I)$ cover $\{x\in BG\langle \AAF\rangle|\hspace{0,1cm}H(x)\leq B\}$, when $I$ runs in the set of ideals of~$\OO_F$ with $N(I)<B^2$. Indeed, let $x\in BG\langle \AAF\rangle$ such that $H(x)\leq B$. Let as define $r_v$ for every $v$ by $H_v(x_v)=q_v^{r_v}$. Define $b_v:=\lceil r_v\rceil.$ Then the ideal $I:=\prod_v \mathfrak p_v^{b_v}$, where~$\mathfrak p_v$ corresponds to~$v$, satisfies that $x\in V(I)$ and its norm is $$N(I)\leq N\big(\prod_{v} \mathfrak p_v^{r_v+1}\big)\leq N\big(\prod_{v}\mathfrak p_v^{2r_v}\big)=B^2.$$  Let~$C$ be such that for all~$v$ and all $r>0$, one has that $$\# c^{-1}(r)\langle\FF_{q_v}\rangle\leq Cq_v^{\dim(c^{-1}(r))+r/2},$$then also $$\# c^{-1}(r)\langle\FF_{q_v}\rangle\leq Cq_v^{(a(c)+1/2)r}.$$ We note that then for all~$v$ and all $r\geq 1$, one has that $$\bigcup_{j=0}^{r}\# c^{-1}(j)\langle\FF_{q_v}\rangle\leq \sum_{j=0}^{r}Cq_v^{(a(c)+1/2)j}\leq Cq_v^{(a(c)+1/2 )(r+1)}\leq Cq_v^{(2a(c)+1)r} .$$One has that \begin{multline*}\mu(V(I))\leq \prod_{\alpha_v(I)>0}C q_v^{(2a(c)+1){\alpha_v(I)}}=C^{\#\{v:\alpha_v(I)>0\}}N(I)^{2a(c)+1}\ll B N(I)^{2a(c)+1}\\\leq B^{2a(c)+2}.\end{multline*}
 There are no more than $O(B^2)$ ideals of norm at most~$B^2$.  The statement follows.
\end{proof}
	\begin{thm}\label{strsuitthm}
		One has that $$\#\{x\in BG\langle F\rangle|\hspace{0,1cm}H(x)\leq B\}\asymp B^{a(c)}\log(B)^{b(c)-1}.$$ 
	\end{thm}
	\begin{proof}
		Clearly, one can replace $BG\langle F\rangle$ in the statement by $i(BG\langle F\rangle)$. 	Let us denote by~$D$ the group $i(BG\langle F\rangle)K$. The group~$D$ is of finite index in~$BG\langle \AAF\rangle$ and let $y_1=1, y_2\doots y_r\in BG\langle\AAF\rangle$ be the representatives modulo~$D$. Let us write $$A:=\bigcup_{i=1}^ry_ii(BG(F)).$$ Observe that \begin{align*}\#\{x\in A|\hspace{0,1cm}H(x)\leq B\}
		=\sum_{i=1}^n\#\{x\in i(BG(F))|\hspace{0,1cm}H(y_ix)\leq B\}.
	\end{align*}
By the fact that $c$ is $m_G$-approachable, Lemma~\ref{hrprime} gives that $x\mapsto H(y_ix)/H(x)$ are bounded functions from below and above by strictly positive constants.	By Theorem~\ref{countx}, it suffices to show that $$\# \{x\in A|\hspace{0,1cm}H(x)\leq B\} \asymp B^{a(c)}\log(B)^{b(c)-1}.$$
By Proposition~\ref{northcott}, one has that $$Z(s):=\sum_{x\in \cup_{j=1}^n y_ji(BG\langle F\rangle)}H(x)^{-s}$$
 converges for $\Re(s)\gg 0$. 
As $H$ is $K$-invariant, we get $$\sum_{x\in \cup _{j=1}^ny_ji(BG\langle F\rangle)}H(x)^{-s}=\int_{BG\langle \AAF\rangle}H^{-s}\mu.$$ 
It is immediate that the claim is valid for~$c$, then it is valid for~$Mc$. We can thus assume that $c(\Delta_G)\subset\ZZ$. For $n\in\NN_0$, we denote by~$a_n$ the number of $x\in A$ such that $H(x)=q^n.$ Let us introduce variable $t=q^{-s}$. We define power series $Q(t):=\sum_{n=0}^\infty a_nt^n$ so that $$Q(t)=Q(q^{-s})=Z(s).$$
We have that $$\int_{BG\langle\AAF\rangle}H^{-s}\mu=Z(s)$$is holomorphic for $\Re(s)>a(c)$, extends to a meromorphic function in the domain $\Re(s)> a(c)-\varepsilon_0$ for some $\varepsilon_0$ with a pole at $s=a(c)$ which is of order~$b(c)$ and no other poles with $$-\frac{\pi}{\log(q)}\leq \Im(s)<\frac{\pi}{\log(q)}$$ in the domain $\Re(s)>a(c)-\varepsilon_0$.
We deduce that the series $Q(t)$ converges absolutely for $|t|\leq q^{-a(c)}$ to a holomorphic function and that this function extends to a meromorphic function on the disk $|t|<q^{-a(c)}+\varepsilon$ for some $\varepsilon>0$ with only pole at $t=q^{-a(c)}$ of order $b(c)$. By \cite[Theorem A.5(c)]{LAGEMANN2015288} there exists $e_0\in\NN_0,$ $\ell\in\NN$ and $C>0$ such that whenever $M\in\{e_0+k\ell|\hspace{0,1cm}k\in\NN_0\}$ one has that 
\begin{align*}\sum_{n=0}^Ma_n&\sim_{M\to\infty} Cq^{Ma(c)}\log(q^{Ma(c)})^{b(c)-1}\\&\sim_{M\to\infty}Cq^{Ma(c)}M^{b(c)-1}\log(q^{a(c)})^{b(c)-1}.
	\end{align*} In particular
	$$ \sum_{n=0}^Ma_n\asymp_{B\to\infty}  q^{Ma(c)}M^{b(c)-1}\log(q^{a(c)})^{b(c)-1}$$ 
	Finally, $B\asymp_{B\to\infty}q^{\lfloor \log_q(B)\rfloor}$ and 
	\begin{align*}
  	\#\{x\in A|\hspace{0,1cm}H(x)\leq B\}
  	&=\sum_{n=0}^{\lfloor\log_q(B) \rfloor}a_n\\
  	&\asymp_{B\to\infty} q^{ \lfloor\log_q(B)\rfloor a(c)}\lfloor\log_q(B)\rfloor^{b(c)-1}\\
  	&\asymp_{B\to\infty} B^{a(c)}\log(B)^{b(c)-1}.
  \end{align*}
%
The statement follows.
	\end{proof}
\subsubsection{}Let us now treat the case when $c$ is assumed only to be suitable.
\begin{lem}\label{ssuitable}
Let $\varepsilon>1$. There exists $\varepsilon_0\in ]1,\varepsilon[\cap \QQ$ 
and $r_0\in c(\Delta_G)-\{0\}$ such that $c':\Delta_G\to\RR_{\geq 0}$ defined by $c'(x)=c(x)$ if $c(x)\neq r_0$ and $c'(x)=(\varepsilon_0)^{-1}r_0$ if $c(x)=r_0,$ is strongly suitable.
\end{lem}
\begin{proof}
It is clear that as $c(\Delta_G)-\{0\}$ has a minimal value,  the supremum $$\sup_{r\neq 0}\frac{1+\dim(c^{-1}(r))}{r}$$ is a positive real number. There exists $\varepsilon'>0$ and~$r_0\in c(\Delta_G)-\{0\}$ such that $$\varepsilon'+ \sup_{r\neq 0}\frac{1+\dim (c^{-1}(r))}{r} <\varepsilon\cdot \frac{1+\dim(c^{-1}(r_0))}{r_0}.$$
By the fact that $c(\Delta_G)$ is countable, we can choose rational $\varepsilon>\varepsilon_0>1$ such that $\varepsilon_0^{-1}r_0\not\in c(\Delta_G)$ and that the above inequality holds for $\varepsilon$ replaced by $\varepsilon_0$. The function $$c':x\mapsto\begin{cases}
	c(x)&\text{ if $c(x)\neq r_0$}\\
	\varepsilon_0^{-1}r_0&\text{ if $c(x)=r_0$}
\end{cases}$$ then satisfies that the supremum of $\frac{1+\dim ((c')^{-1}(r))}{r}$ is attained for $r=\varepsilon_0^{-1}r_0$. Moreover, for any $r_1\not\in\{ \varepsilon^{-1} r_0,0\}$ it is clear that \begin{multline*}\frac{1+\dim((c')^{-1}(\varepsilon_0^{-1} r_0))}{(\varepsilon_0)^{-1}r_0}-\frac{1+\dim ((c')^{-1}(r_1))}{r_1}\\=\frac{1+\dim((c)^{-1}( r_0))}{(\varepsilon_0)^{-1}r_0}-\frac{1+\dim (c^{-1}(r_1))}{r_1}>\varepsilon'.
	\end{multline*}
	The claim has been verified.
\end{proof}
\begin{cor}
For every height~$H$ defined by a suitable counting function $c:\Delta_G\to\RR_{\geq 0}$ and every $\varepsilon>1$, there exists a height $H'$, defined by a strongly suitable counting function, satisfying that $H\geq H'\geq H^{1/\varepsilon}$ is bounded above by~$1$.
\end{cor}
\begin{proof}
Let $c'$ be defined as in Lemma~\ref{ssuitable}. Then $H'$ defined by~$c'$, satisfies that $H\geq H'$ and that $(H')^{\varepsilon}\geq H$.
\end{proof}
\begin{cor}\label{ainvariantcorrect}
Let $H$ be a suitable counting function. For any $\varepsilon>0$, one has that $$B^{a(c)} \ll \#\{x\in BG\langle F\rangle|\hspace{0,1cm}H(x)\leq B\}\ll B^{a(c)+\varepsilon}.$$
\end{cor}

\subsubsection{}\label{ExamplesFromReps}
Examples of strongly suitable counting functions come from representations of the cyclic group $G=\ZZ/p\ZZ$ of order $p$, where $p$ is the characteristic of the ground field. For $ 1 \le i \le p$, let $V_i$ denote the unique $i$-dimensional indecomposable linear representation of $G$ over $\FF_q$ and define the representation $V=\bigoplus_{\lambda=1}^l V_{d_\lambda}$, $1\le d_\lambda \le p$ of dimension $d=\sum_{\lambda=1}^l d_\lambda$. Following \cite{Yasuda_2014}, we define its invariant $D_V$ by 
\[
D_V := \sum_{\lambda=1}^l \frac{(d_\lambda -1)d_\lambda}{2},
\]
and for a positive integer $j$ coprime to $p$,  define 
\[
 \mathrm{sht}_V(j) :=\sum_{\lambda=1}^l \sum_{i=1}^{d_\lambda -1} \left \lfloor \frac{ij}{p} \right \rfloor. 
\]
(This is  $-\mathbf{w}$ with the notation from \cite{woodyasuda}.)
The $\mathbf{v}$-function on $\Delta_G$ associated to $V$ \cite[Definition 3.7]{woodyasuda} is then defined by 
\[
 \mathbf{v(x)} =
 \begin{cases}
  0 & (x=0) \\
  d-l + \mathrm{sht}_V(j) & (\text{$x$ has ramification jump $j$}).
 \end{cases}
\]
Assuming that $D_V\ge p$, we will show that $\mathbf{v}$ is a strongly suitable raising function and compute $a(\mathbf{v})$ and $b(\mathbf{v})$ below. 
Examples of representations $V$ with $D_V=p$ include:
\begin{itemize}
\item $V_2^{\oplus p}$ ($\forall p$),
\item $V_2^{\oplus 2}$ ($ p = 2 $),
\item $V_3$ ($ p = 3 $),
\item $V_3\oplus V_2^{\oplus 2}$ ($p=5$),
\item $V_4 \oplus V_2$ ($p=7$).
\end{itemize}

For $j>0$ with $p\nmid j$, 
 the locus $C_j \subset \Delta_G$ of points with  ramification jump $j$  has coarse moduli space which is the ind-perfection of $\mathbb{G}_m \times \mathbb{A}^{j-\lfloor j/p\rfloor -1 }$. In particular, it has dimension $j-\lfloor j/p \rfloor$. 

\begin{lem}
If $D_V=p$, then for every $j>0$ with $p\nmid j$, we have
\[
 \dim C_j -\mathbf{v}(j) \le -1. 
\]
Morever, we have 
\[
\dim C_{p-1} -\mathbf{v}(p-1) = -1. 
\]
\end{lem}
\begin{proof}
We have 
\[
\dim C_{j} -\mathbf{v}(j) = 
j-\left\lfloor \frac{j}{p}\right\rfloor
+l-d -\mathrm{sht}_V(j).
\]
From \cite[Proof of Proposition 6.9]{Yasuda_2014}, $\mathrm{sht}_V(j+p)$=$\mathrm{sht}_V(j)+p$. 
Thus, 
\begin{align*}
\dim C_{j+p} -\mathbf{v}(j+p)& = 
j+p-\left(\left\lfloor \frac{j}{p}\right\rfloor +1 \right)
+l-d -(\mathrm{sht}_V(j)+p)\\
&=\dim C_{j} -\mathbf{v}(j)-1.
\end{align*}
Therefore, $\dim C_{j} -\mathbf{v}(j)$ attains the maximum only on $\{1,2,\dots,p-1\}$. 
From \cite[Lemmas 3.2 and 3.3]{Yasuda_discrepancies},  for $j\in \{1,2,\dots,p-1\}$, we have
\begin{align*}
&\dim C_{j} -\mathbf{v}(j)\\
&= l-d+j-\mathrm{sht}_V(j) \\
&= \mathrm{sht}_V(p-j)+j-p  \\
&\le (p-j-1)+j-p  \\
& = -1.
\end{align*}
We have proved the inequality of the lemma.
For $j=p-1$, we have
\begin{align*}
	\dim C_{p-1} -\mathbf{v}(p-1)
 = \mathrm{sht}_V(1)-1 = -1. 
\end{align*}
\end{proof}

\begin{lem}
We have 
\begin{align*}
a(\mathbf{v})&= 1,\\
b (\mathbf{v})&= \# \{ s \in \{1,2,\dots, p-1\} \mid l-d+s-\mathrm{sht}_V(s)= -1 \}.
\end{align*}
\end{lem}
\begin{proof}
For every $r>0$, we have the equivalence
\[
 \frac{1+\dim \mathbf{v}^{-1}(r)}{r} \le 1 \Leftrightarrow \dim \mathbf{v}^{-1} (r) - r \le -1.
\]
Moreover, the equality on the left is strict if and only if so is the one on the right. 
From  the last lemma, the right inequality holds for every $r=\mathbf{v}(j)$ and the equality holds for $r=\mathbf{v}(p-1)$. This shows that $a(\mathbf{v})=1$.
Moreover, from the proof of the last lemma, the equality $\mathbf{v}^{-1} (r) - r = -1$ fails for $j>p$. This shows the assertion for $\mathbf{v}$.
\end{proof}

\begin{lem}
Suppose $D_V=p$. Then, $\mathbf{v}$ is strongly suitable. 
\end{lem}

\begin{proof}
(1) and (3) of Definition \ref{suitable} are obvious. (2) follows from the fact that 
the coarse moduli space of $c^{-1}(r)$ is an open subset of an affine space.

We can show that $\mathbf{v}$ is $m$-approachable for any $m$ by using the fact that
\[
k\llparenthesis t \rrparenthesis / \wp k\llparenthesis t \rrparenthesis =\bigoplus _{j>0;p\nmid j} kt^{-j} \oplus k/\wp k
\] 
and a Laurent polynomial in $\bigoplus _{j>0;p\nmid j} kt^{-j} \oplus k/\wp k$ of order $-j$ corresponds to a torsor of ramification jump $j$. 

Lastly, we show that $\mathbf{v}$ is strongly suitable. In the current situation, for $r'=\mathbf{v}(j)$, the inequality in Definition \ref{strongsuitdef} is equivalent to
\[
\dim C_j -\mathbf{v}(j)= \dim \mathbf{v}^{-1}(r')-r' < -1 -\varepsilon r'= -1 -\varepsilon \mathbf{v}(j).
\]
To have this inequality, it is enough to put  $\varepsilon $ to be a positive real number less than
\[
 \min \left \{\frac{1}{\mathbf{v}(j')+1} \mid 1\le j' \le p-1 \right \}. 
\]
 Indeed, for $1\le j \le p-1$ with $\dim C_j -\mathbf{v}(j)<-1$, we have
\[
\dim C_j-\mathbf{v}(j) \le -2 < 
-1 - \varepsilon \mathbf{v}(j) .
\]
For $j>p$, we write $j=j'+np$, $1\le j' \le p-1$. Then, since for $n \ge 1$, 
\[
 \frac{1}{\mathbf{v}(j')+1} \le  \frac{n}{\mathbf{v}(j')+n} = \frac{n}{\mathbf{v}(j)},
\]
we get
\begin{align*}
\dim C_j-\mathbf{v}(j)
&= \dim C_{j'}-\mathbf{v}(j')-n \\
&\le -1 -n  \\
&=  -1 - \frac{n}{\mathbf{v}(j)}\mathbf{v}(j) \\
& <  -1 - \varepsilon \mathbf{v}(j).
\end{align*}
\end{proof}

\appendix
\section{Appendix: a counting result}
\begin{thm}\label{countx}
Suppose that $X$ is a discrete set. For $i=1\doots n$, where $n\in\NN$, let $H_i:X\to\RR_{>0}$ be height functions satisfying that:
\begin{enumerate}
	\item For any $1\leq i, j\leq n$, there exists $C_{ij}>0$ such that $$C_{ij}^{-1}H_j(x)\leq H_i(x)\leq C_{ij}H_j(x) $$ for all $x\in X$.
	\item One has that $$\sum_{i=1}^n\#\{x\in X|\hspace{0,1cm}H_i(x)\leq B\} \asymp B^a\log(B)^b$$ for some $a>0$ and $b\geq 0$.
\end{enumerate}
Then for every $i=1\doots n$, one has that $$\#\{x\in X|\hspace{0,1cm}H_i(x)\leq B\}\asymp B^a\log(B)^b.$$
\end{thm}
\begin{proof}
We choose $i=1$ and write $C_j$ for $C_{1j}$, when $j\geq 2$. Set $C_1=1$. Let $C:=\max C_j$. Let us write $$N_j(B):=\#\{x\in X|\hspace{0,1cm}H_j(x)\leq B\}.$$ We verify that $N_1(B)\asymp B^a\log(B)^b.$ Let $A>0$ be such that  $$\sum_{i=1}^nN_i(B)\leq AB^a\log(B)^b$$ for all $B\geq B_0$, where $B_0>0$.  One has that 
\begin{align*}
\sum_{i=1}^nN_i(B)&=\sum_{i=1}^n\#\{x\in X|\hspace{0,1cm}H_i(x)\leq B\}\\
&\geq\sum_{i=1}^n\#\{x\in X|\hspace{0,1cm}CH_1(x)\leq B\}\\
&=\sum_{i=1}^n\#\{x\in X|\hspace{0,1cm}H_1(x)\leq C^{-1}B\}\\
&=n \cdot \#\{x\in X|\hspace{0,1cm}H_1(x)\leq C^{-1}B\}\\
&=n N_1(C^{-1}B).
\end{align*}
We deduce that $$ N_1(C^{-1}B)\leq \frac{1}{n} \sum_{i=1}^nN_i(B)\leq AB^a\log(B)^b\leq 2^{b} AC^a\cdot (C^{-1}B)^a\log(C^{-1}B)^b $$ for $B$ sufficiently large. Thus $N_1(B)\ll B^a\log(B)^b$ and analogously one verifies that $N_1(B)\gg B^a\log(B)^b$. The statement is proven. 
\end{proof}
\section{Appendix: A result on commutative group stacks}
Let~$G_0$ and~$G_{-1}$ be commutative group schemes over a field~$k$, which are isomorphic to affine spaces as schemes.  Let $\phi:G_{-1}\to G_0$ be a $k$-homomorphism. We let $G_{-1}$ acts on $G_0$ via~$\phi$. For a $k$-ring~$A$, giving a $G_{-1}$-equivariant morphism $G_{-1}\times_kA\to G_0\times_kA,$ where the actions of~$G_{-1}$ are via the first coordinate, is equivalent to giving an element of~$G_0(A)$. Consider a homomorphism $f:A\to B.$ Let $g:G_{-1}\times_kA\to G_0\times_kA $  and $h:G_{-1}\times _kB\to G_0\times_kB$ be given $G_{-1}$-equivariant homomorphisms. Giving a~$G_{-1}$-equivariant morphism $t: G_{-1}\times _kB\to (G_{-1}\times_kA)\times_AB=G_{-1}\times_kB$ such that $t\circ(g\times_AB)=h$ is equivalent to giving  an element $t\in G_0(B)$ such that $tg|_B=h$.
\begin{prop}\label{gogi}
Consider the category of affine schemes $\Aff_k$ over~$k$, endowed with fppf topology. Consider the full subcategory $(G_{0}/G_{-1})^*$ of the stack $G_{0}/G_{-1}$ on $\Aff_k$, having for objects over a ring~$A$ the elements of~$G_0(A)$. 
One has that $(G_0/G_{-1})^*$ is a commutative group stack and the canonical inclusion is an equivalence $(G_0/G_{-1})^*\xrightarrow{\sim}G_0/G_{-1}.$
\end{prop}
\begin{proof}
It is clear that the morphisms over~$\phi:A\to B$ in $(G_0/G_{-1})^*$ are elements~$t$ as above. By Lemma~\ref{awnz}, any object in $(G_0/G_{-1})(A)$ is isomorphic to an object in $(G_0/G_{-1})^*(A).$ Any morphism in $(G_0/G_{-1})$ is isomorphic to the one given by some element~$t$. We deduce that the inclusion $(G_0/G_{-1})^*\to (G_0/G_{-1})$ is equivalence of categories fibered in groupoids over~$\Aff_k$.  Now it remains to see that $(G_0/G_{-1})^*$ is preserving the addition, the commutativity and associativity functors. The sum of a morphisms $1\mapsto g$ and $1\mapsto g'$ is the morphism $1\mapsto g+g'$, hence is preserved. The commutativity and associativity functors \cite[Paragraph 1.4.11]{Dualite} are given by element~$0$, hence are preserved. The proof is completed.  
\end{proof}
\bibliographystyle{alpha}
\bibliography{bibliography}
\end{document}